\documentclass[dvips,12pt,aop]{imsart}

\usepackage{amsthm,amsmath,amsfonts,amssymb}

\startlocaldefs

\numberwithin{equation}{section}
\newcommand{\be}{\begin{eqnarray}}
\newcommand{\e}{\end{eqnarray}}
\newcommand{\bex}{\begin{eqnarray*}}
\newcommand{\ex}{\end{eqnarray*}}
\newcommand{\mbb}{\mathbb}

\newcommand{\mbf}{\mathbf}
\newcommand{\mcal}{\mathcal}
\newcommand{\la}{\langle}
\newcommand{\ra}{\rangle}

\newtheorem{theo}{Theorem}
\newtheorem{defi}{Definition}
\newtheorem{lem}{Lemma}
\newtheorem{prop}{Proposition}
\newtheorem{rem}{Remark}
\newtheorem{cor}{Corollary}

\endlocaldefs

\begin{document}

\begin{frontmatter}

\title{Branching Brownian motion in a strip: Survival near criticality}
\runtitle{Branching Brownian motion in a strip}


\author{\fnms{S. C.} \snm{Harris}\ead[label=e1]{m.hesse@bath.ac.uk}\thanksref{t1}}
\author{\fnms{M.} \snm{Hesse}\ead[label=e2]{s.c.harris@bath.ac.uk}\thanksref{t3}}
\and
\author{\fnms{A. E.} \snm{Kyprianou}\ead[label=e3]{a.kyprianou@bath.ac.uk}\thanksref{t2}}
\thankstext{t2}{Corresponding author. Email: a.kyprianou@bath.ac.uk}
\thankstext{t1}{Email: s.c.harris@bath.ac.uk.}
\thankstext{t3}{Email: m.hesse@bath.ac.uk.}

\affiliation{University of Bath}

\runauthor{Harris, S.C., Hesse, M. and Kyprianou, A.E.}

\begin{abstract}
We consider a branching Brownian motion with linear drift in which particles are killed on exiting the interval $(0,K)$ and study the evolution of the process on the event of survival as the width of the interval shrinks to the critical value at which survival is no longer possible. We combine spine techniques and a backbone decomposition to obtain exact asymptotics for the near-critical survival probability. This  allows us to deduce the existence of a quasi-stationary limit result for the process conditioned on survival which reveals that the backbone thins down to a spine as we approach criticality.\\
This paper is motivated by recent work on survival of  near critical  branching Brownian motion with absorption at the origin by A\"id\'ekon and Harris in \cite{aidekonharris} as well as the  work of  Berestycki et al. in \cite{bbs} and \cite{bbs2}.
\end{abstract}



\end{frontmatter}


\section{Introduction and main results}
\subsection{Introduction and main results}\label{sec-introduction}
We consider a branching Brownian motion in which each particle performs a Brownian motion with drift $-\mu$, for $\mu\geq 0$, and is killed on hitting $0$ or $K$. All living particles undergo branching at constant rate $\beta$ to be replaced by a random number of offspring particles, $A$, where $A$ is an independent random variable with distribution $\{ q_k; k= 0,1,...\}$, finite mean $m>1$ and such that $E(A\log^+A)<\infty$. 
Once born, offspring particles move off independently from their birth position, repeating the stochastic behaviour of their parent.\\ 
In other words, the  motion of a single particle is governed by the infinitesimal generator 
 \be\label{eq-generatorl}
L=\frac{1}{2} \frac{d^2}{dx^2} -\mu \frac{d}{dx}, \ \ x\in(0,K),  
\e
defined for all functions $u\in C^2(0,K)$, the space of twice continuously differentiable functions on $(0,K)$, with u$(0+)=u(K-)=0$. 
The branching activity is characterised by the branching mechanism 
\be\label{eq-defbranchingmechf}
F (s) = \beta(G(s)-s), \ \ \ s\in[0,1],  
\e 
where $G(s)=\sum_{k=0}^\infty q_k s^{k}$  is the probability generating function of $A$.\\
Let us introduce some notation. Denote by $N_t$ and $|N_t|$ the set of and the number of particles alive at time $t$ respectively. For a particle $u \in N_t$, we write  $x_u(t)$ for its spatial position at time $t$. We define $X_t=\sum_{u\in N_t} \delta_{x_u(t)}$ to be the spatial configuration of  particles alive at time $t$ and we set $X=(X_t, t\geq 0)$. 
Denote by $P_\nu^K$  the law of $X$ with $X_0=\nu$ where  $\nu\in\mcal{M}_a(0,K)$, the space of finite atomic measures on $(0,K)$ of the form $\sum_{i=1}^n \delta_{x_i}$ with $x_i\in(0,K)$ and $n\in\mbb{N}$. 
If the process is initiated from a single particle at $x\in(0,K)$, then we simply write $P_x^K$ (instead of $P_{\delta_x}^K$). We will sometimes neglect the dependence on the initial configuration and write $P^K$ without a subscript.  We call the process $X$ a $P^K$-branching diffusion. \\ 
Further, $(\xi=(\xi_t, t\geq 0), \mbb{P}^K_x)$ will henceforth denote a Brownian motion with drift $-\mu$ starting from $x\in(0,K)$ which is killed upon exiting the interval $(0,K)$. $\mbb{P}^K$ is the law of the single particle motion under $P^K$. 
\\
For $x\in[0,K]$ we  define the survival probability  $p_K(x)=P_x^K(\zeta=\infty)$ where $\zeta=\inf\{t>0: |{N}_t|=0\}$ is the time of extinction. 
As a first result we identify the critical width $K_0$ below which survival is no longer possible. 
\begin{theo}\label{theo-criteriaforpossurvivalprob}
 If  $\mu<\sqrt{2(m-1)\beta}$ and  $K> K_0$ where $K_0:= {\pi}( { \sqrt{2(m-1)\beta-\mu^2}})^{-1}$, then  $p_K(x)>0$ for all $x\in(0,K)$; 
otherwise $p_K(x)=0$ for all $x\in[0,K]$.
\end{theo}
The proof of Theorem \ref{theo-criteriaforpossurvivalprob}  uses a spine argument,  decomposing $X$ into a Brownian motion  conditioned to stay in $(0,K)$ dressed with independent copies of $(X,P^K)$ which immigrate along its path. \\
As we want to study the evolution of the $P^K$-branching diffusion on survival, we will develop a decomposition which identifies the particles with infinite genealogical lines of descent,  that is, particles which produce a family of descendants which survives forever. 
To illustrate this, in a realisation of $X$, let us colour blue all particles with an infinite line of descent and colour red all remaining particles. 
Thus, on the event of survival, the resulting picture consists of a blue tree `dressed' with red trees whereas, on the event of extinction, we see a red tree only. \\
The  branching rates of the branching diffusions corresponding to the blue and red trees can be intuitively derived as follows. For simplicity, consider the dyadic branching case only. A particles dies and is replaced by two offspring, at position $x$ say, at rate $\beta$. The probability that one of its offspring has an infinite genealogical line of descent is $p_K(x)$, independent of the other particle. Thus, with probability $p_K(x)^2$ both offspring particles are blue and likewise with probability $(1-p_K(x))^2$ and probability $2 p_K(x)(1-p_K(x))$ two red ones, respectively one blue and one red particle are born. Thus, given a particle is blue, it branches into two blue particles at rate $\beta \frac{p_K(\cdot)^2}{p_K(\cdot)}= \beta p_K(\cdot)$ and, given a particle is red, it  branches into two red particles at rate $\beta (1-p_K(\cdot))$ while, given a particle is blue,  immigration of a red particle occurs at rate $2\beta (1-p_K(\cdot))$. Similar reasoning gives the result for the general branching mechanism case, see Section \ref{sec-backbone}. There we shall also see that particles in branching diffusion corresponding to the red trees, respectively the blue tree, move according to a Brownian motion with drift $-(\mu+\frac{p_K'}{1-p_K})$, respectively $-(\mu-\frac{p_K'}{p_K})$. This results from $h$-transforms of $L$ using $h=1-p_K$ and $h=p_K$ respectively. In fact we will show that the laws of  the branching diffusions corresponding to the red tree, the blue and the dressed blue tree arise from martingale changes of measure which, on the level of infinitesimal generators, correspond to the aforementioned $h$-transforms. \\
Suppose we know the branching mechanisms of the branching diffusions corresponding to the blue and the red trees in the general case as well as the immigration rates (we will see that a branching mechanism of the general form  induces a second type of immigration at the branching times of the blue tree).  Intuitively speaking, the coloured tree starting from $x\in(0,K)$ is then constructed by  flipping a coin with probability $1-p_K(x)$ of `heads' and if it lands `heads' we grow a red tree with initial particle at $x$, while if it lands `tails' we grow a blue tree at $x$ and dress its branches with red trees. 
Let us write $\mbf{P}^{K}_x$ for the law of the coloured tree which is defined on the filtration $\mcal{F}_t^c:=\sigma\{\mcal{F}_t, c(u)_{u\in N_t}\}$, where $(\mcal{F}_t,t\geq 0)$ is the natural filtration of $(X,P^K)$ and $c(u)$ is the colour of particle $u\in N_t$. Note that the colours of all particles are $\mcal{F}_\infty$-measurable. Then we can state (what will turn out to be a simplified version of) the so-called backbone decomposition.
\begin{theo}[Backbone decomposition]\label{theo-backboneintro}
Let $K>K_0$ and $x\in(0,K)$. On the filtration $(\mcal{F}_t, t\geq 0)$ (which means we ignore the colouring),  $(X,\mbf{P}^K_x)$ is equal in law to  $(X, P^K_x)$. This is, for all $t\in[0,\infty]$ and $A\in\mcal{F}_t$, we have $\mbf{P}^K_x(A)=P_x^K(A)$.
\end{theo}
An alternative statement of this result is given as  
Theorem \ref{cor-backbonedecomposition} in Section \ref{sec-backbone}. 
There we show that  the backbone decomposition arises naturally from combining changes of measure which condition $(X,P^K)$ on either the event of survival or the event of extinction.
A significant convenience of the backbone decomposition is that  conditioning the $P^K$-branching diffusion on survival is the same as conditioning on there being a  dressed blue tree, that is a blue tree `dressed' with red trees. Thus, instead of studying the quasi-stationary limit $\lim_{K\downarrow K_0}P^K_x(\cdot|\zeta=\infty)$ it suffices to study the evolution of the branching diffusion corresponding to a dressed blue tree as $K\downarrow K_0$.\\ 
In order to do this, we need to know the asymptotics of the survival probability $p_K$ near criticality. 
For a first asymptotic result note that    $u=1-p_K$ solves the differential equation $L u+F(u)=0$ on $(0,K)$ with boundary condition $u(0)=u(K)=1$ (cf. Remark \ref{rem-fkpp}). Near criticality we may assume  
that $p_K(x)$ is very small for a fixed $x$ and neglecting all terms of order $(p_K(x))^2$ and higher we obtain the linearisation $L p_K + m \beta p_K = 0$. This  suggests $p_K(x)\sim {{C}}_K \sin(\pi x/K_0) e^{\mu x}$. In fact we have the following result. 
\begin{theo}\label{theo-survivalprobsonstant}\label{theo-survivalprob}\label{theo-survivalprobrough}
 Uniformly for all $x\in(0,K_0)$, we have
\be\label{eq-firstbit}
p_K(x) \sim {C}_K \sin(\pi x /K_0) e^{\mu x}, \ \ \ \ \text{as} \  K\downarrow K_0,
\e
where ${C}_K$ is independent of $x$ and can explicitly be determined  as 
\bex
{C}_K =  {  (K-K_0)}   \frac{(K_0^2\mu^2+\pi^2)(K_0^2 \mu^2 + 9 \pi^2)}{12 (m-1)\beta \pi  K_0^3 (e^{\mu K_0}+1)}, \ \ \ \ \text{as} \ \ K\downarrow K_0,
\ex
and in particular  ${C}_K\downarrow 0$ as $K\downarrow K_0$. 
\end{theo}
In Section \ref{sec-proofroughprob} we will prove a first part of Theorem \ref{theo-survivalprob}, that is equation (\ref{eq-firstbit}), in the fashion of \cite{aidekonharris} using  spine techniques. It is to be particularly emphasized that we are able to determine $C_K$ here.  ${C}_K$ is a `non-linear' constant and  `linear' spine techniques fail when trying to identify it. 
Nevertheless, a careful application of the backbone decomposition given in Theorem \ref{theo-backboneintro}  will deliver an explicit expression here since the blue tree captures enough `non-linear' branching information about the evolution of $(X,P^K)$ on survival. A heuristic argument is given in Section \ref{sec-heuristics} followed by the rigorous proof in Section \ref{sec-rigorous}.
\\
With Theorem \ref{theo-backboneintro} and \ref{theo-survivalprobsonstant} in hand we look for a quasi-stationary limit result for the law of the branching diffusion corresponding to a  dressed blue tree, which agrees with the law of $X$ conditioned on survival, as we approach criticality.
In the dyadic case, the heuristic derivation of the branching rates already suggests that, given the particle positions $x_u(t)$ for $u\in N_t$ in $X_t$, the number of particles in the blue tree at time $t$, is the number of successes in a sequence of independent Bernoulli trials each with  probability of success $p_K(x_u(t))$, $u\in N_t$ (We will address this thinning argument rigorously in Remark \ref{rem-thinning}). Now, as $K\downarrow K_0$, the probabilities $p_K(\cdot)$ tend to $0$ uniformly by Theorem \ref{theo-survivalprobrough} and thus the blue tree  becomes increasingly thinner on $(0,K_0)$. Under conditioning on survival, it cannot vanish completely though since the genealogical line of the initial blue particle cannot become extinct and thus one may believe that,  over a fixed time interval $[0,T]$,  the blue tree thins down to a single genealogical line  at criticality. In the case of a dyadic branching mechanism this conjecture can readily by confirmed by looking at the branching rates. The blue branching rate $\beta p_K$ drops down to $0$ as $K\downarrow K_0$, at the same time the red branching rate $\beta (1-p_K)$ increases to $\beta$ and the rate of immigration $2\beta (1-p_K)$ rises to $2\beta$ at criticality. Formalising this idea and taking into account the change in the single particle motion, the general results reads  as follows.
\begin{theo}\label{theo-quasistationary}
Let $x\in(0,K_0)$. Consider a process $X^*=(X_t^*,t\geq 0)$ which evolves as follows.  $X^*$ is initiated from a single particle at $x$ performing a Brownian motion conditioned  to stay in $(0,K_0)$, i.e. a strong Markov process with infinitesimal generator
\be\label{eq-operatorconditioned}
L^*_{K_0} = \frac{1}{2}\frac{d^2}{dx^2} +  \frac{ \pi/K_0 }{ \tan(\pi x/K_0)} \frac{d}{dx},
\e
defined for all $u\in C^2(0,K_0)$.
Along its path we immigrate $\tilde{A}$ independent copies  of $(X,P^K)$ at rate $m\beta$ where $\tilde{A}$ has the size-biased offspring distribution $(\tilde{q}_k, k=0,1,...)$ with 
\bex
\tilde{q}_k = q_{k+1} \frac{k+1}{m}, \ \ \ k\geq 0. 
\ex
Denote the law of $X^*$ by $P_x^*$.
Then, for any fixed time $T> 0$, the law of $(X_t,0\leq t\leq T)$ under the measure $\lim_{K\downarrow K_0} P_x^K(\cdot | \zeta=\infty)$ is equal to $(X_t^*,0\leq t\leq T)$ under $P^*_x$.
\end{theo}
The construction of the backbone via a martingale change of measure allows us to give a very simple proof of this quasi-stationary limit result. Theorem \ref{theo-quasistationary} can be seen as an extension of the spine decomposition we mentioned  in the discussion following Theorem \ref{theo-criteriaforpossurvivalprob} to the critical width $K_0$. We emphasize however that the result, as stated, only holds over finite time horizons $[0,T]$. \\
In Section \ref{sec-superbrownianmotion}, we demonstrate the robustness of our approach by applying the results for the $P^K$-branching diffusion to study the evolution of a supercritical super-Brownian motion with absorption at $0$ and $K$ near criticality. We outline a backbone decomposition analogous to Theorem \ref{theo-decomposition} in which we will see that the backbone of the super-Brownian motion with absorption at $0$ and $K$ is the same as the backbone of an associated $P^K$-branching diffusion. This connection allows us to  deduce asymptotic results for the survival rate of the super-Brownian motion with absorption on $(0,K)$ directly from the results on the survival probability of the associated  $P^K$-branching diffusion. Further, we can find a quasi-stationary limit result for the super-Brownian motion equivalent to Theorem \ref{theo-quasistationary}. This section is  intended to highlight the applicability of the backbone approach and we will only sketch the proofs of the results therein.\\
Our paper is organised as follows. In Section \ref{sec-spine} we present the proof of Theorem \ref{theo-criteriaforpossurvivalprob} using spine techniques. In Section \ref{sec-backbone} we show that the backbone arises from a martingale change of measure which conditions $(X,P_\nu^K)$ on survival, and we establish the backbone decomposition. We prove the asymptotic results for the survival probability given in Theorem \ref{theo-survivalprobrough} in  Section \ref{sec-survivalprobpart1}. The proof of the quasi-stationary limit result in Theorem \ref{theo-quasistationary} follows in Section \ref{sec-backbonecriticality}. Section \ref{sec-superbrownianmotion} sketches the analogous results for the super-Brownian motion on $(0,K)$.

\subsection{Literature overview}\label{sec-literature}
Spine techniques of the type used in the proof of Theorem \ref{theo-criteriaforpossurvivalprob} were developed in  Chauvin and Rouault \cite{chauvinrouault}, Lyons \cite{lyons} and Lyons et al. \cite{lyonsetal} and are now a standard approach in the theory of branching processes. See, for example, Harris et al. \cite{hhk} and Kyprianou \cite{kyprianou}   for related applications in the setting of branching Brownian motion with absorption at $0$ respectively absorption at a space-time barrier.

A backbone decomposition as in Theorem \ref{theo-backboneintro} for supercritical superprocesses is presented in Berestycki et al. \cite{bkm}. It extends the earlier work of  Evans and O'Connell \cite{evansoconnell}, Fleischmann and Swart \cite{fleischmannswart04} and  Engl\"ander and Pinsky \cite{englaenderpinsky} as well as the corresponding decomposition for continuous-state branching processes in Duquesne and Winkel  \cite{duquesnewinkel}. \\
The results for superprocesses are complemented by the decomposition in Etheridge and Williams \cite{etheridgewilliams} which considers the $(1+\beta)$-superprocess conditioned on survival.  This work is of particular interest in the current context since it also presents the equivalent result for the approximating branching particle system. However we should point out that in their case the immigrants are conditioned to become extinct up to a fixed time $T$ whereas, in our setting, we condition on extinction in the strip $(0,K)$. Thus the underlying transformations in \cite{etheridgewilliams} are time-dependent in contrast to the space-dependent $h$-transforms we see in our setting. \\
 We also point out that our derivation of the backbone decomposition differs from the previously mentioned articles in that we show that the backbone arises from combining changes of measure which condition $(X,P_\nu^K)$ on either the event of survival or the event of extinction.

The equivalent result to Theorem \ref{theo-survivalprob} in the setting of Branching Brownian motion with absorption at the origin was shown in Berestycki et al. \cite{bbs2} and A\"id\'ekon and Harris  \cite{aidekonharris}.  However, it has not been possible so far to give such an explicit expression for the constant analogous to ${C}_K$. 

A similarly fashioned result to Theorem \ref{theo-quasistationary} was obtained in the aforementioned work by Etheridge and Williams \cite{etheridgewilliams}. 
Their result extends the Evans immortal particle representation for superprocesses in \cite{evans} which is the equivalent of the spine representation for branching processes. Again we point out that, in contrast to our setting,  extinction is  a time-dependent phenomenon. 
Further, our martingale change of measure approach to the backbone decomposition allows us to give a very simple proof of the quasi-stationary limit result.

\section{Spine techniques - Proof of Theorem \ref{theo-criteriaforpossurvivalprob}}\label{sec-spine}

The proof of Theorem \ref{theo-criteriaforpossurvivalprob} uses classical spine techniques developed in Chauvin and Rouault \cite{chauvinrouault}, Lyons et al. \cite{lyonsetal} and Lyons \cite{lyons}; see e.g. Harris et al. \cite{hhk} and Kyprianou \cite{kyprianou} for related applications in the setting of Branching Brownian motion with absorption at $0$.\\ 
We will briefly recall the key steps in the spine construction. For a comprehensive account we refer the reader to Hardy and Harris \cite{hardyharris}.\\
Recall that  we denote by $(\xi,\mbb{P}^K_x)$ a Brownian motion with drift $-\mu$  initiated from $x\in(0,K)$ which is killed upon exiting $(0,K)$. 
Then the process 
\be\label{eq-singleparticlemartingale}
\Upsilon^K(t)=\sin(\pi \xi_t/K)e^{\mu \xi_t+(\mu^2/2+\pi^2/2K^2)t}, \ \ \ t\geq 0,
\e
 is a martingale with respect to $\sigma(\xi_s:s\leq t)$. 
Define $\mbb{Q}_x^K$ to be the probability measure which has martingale density $\Upsilon^K(t)$ with respect to $\mbb{P}_x^K$ on $\sigma(\xi_s: s\leq t)$. Under $\mbb{Q}_x^K$, $\xi$ is now a Brownian motion conditioned to stay in $(0,K)$; cf. Knight \cite{knight}.\\
Using ideas in \cite{hardyharris}, we can  use $\Upsilon^K$ to construct a martingale with respect to $\mcal{F}_t=\sigma(X_s, s\leq t)$, the filtration generated by the $P^K$-branching diffusion up to time $t$. For each $u\in N_t$, write $\Upsilon^K_u(t)=\sin(\pi x_u(t)/K)e^{\mu x_u(t)+(\mu^2/2+\pi^2/2K^2)t}$, $ t\geq 0$.  Define the process $Z^K=(Z^K(t),t\geq 0)$ given by 
\bex
Z^K(t) = \sum_{u\in N_t} e^{-(m-1) \beta t} \Upsilon_u(t)
= \sum_{u\in N_t} e^{\mu x_u(t)-\lambda(K) t} \sin(\pi x_u(t)/K), \  t\geq 0, 
\ex
where we set $\lambda(K):=(m-1)\beta - \mu^2/2 - \pi^2/2K^2$. Then $Z$ is a non-negative $(P^K_x,\mcal{F}_t)$-martingale. For $x\in(0,K)$, we define a martingale change of measure on the probability space of the  $P^K$-branching diffusion via
\be\label{changemeasure-z}
\left.\frac{d {Q}_x^K}{dP_x^K}\right|_{\mcal{F}_t} =\frac{Z^K(t)}{Z^K(0)}.
\e
This change of measure induces the following spine construction for the path of $X$ under ${{Q}}_x^K$. From the initial position $x$, we run a $\mbb{Q}^K_x$-diffusion, that is a  Brownian motion conditioned to stay in $(0,K)$, and we call it a spine. At times of a Poisson process with rate $m\beta$ we immigrate $\tilde{A}$ independent copies of $(X,P^K)$ rooted at the spatial position of the spine at this time. The number of immigrants $\tilde{A}$ has the size-biased offspring distribution 
\bex 
\tilde{q}_k =  \frac{1+k}{m} q_{k+1}, \ \  k\geq 0.
\ex 
Let us remark that the change of measure with the martingale $\Upsilon^K$ in (\ref{eq-singleparticlemartingale}) is equivalent to a Doob's $h$-transform  on $L$ using $h(x)=\sin(\pi x/K) e^{\mu x}$. The infinitesimal generator $L^*$ of a Brownian motion conditioned to stay in $(0,K)$  is therefore
\bex
L^*_K = \frac{1}{2}\frac{d^2}{dx^2} +   \frac{ \pi/K_0 }{ \tan(\pi x/K_0)} \frac{d}{dx},
\ \ \text{ on } \ (0,K),
 \ex
with domain $ C^2(0,K)$. We  then call the Brownian motion conditioned to stay in $(0,K)$ a $L^*_K$-diffusion and note that it is positive recurrent with invariant density $ \frac{2}{K} \sin^2(\pi x /K)$, for $x\in (0,K)$. \\
\begin{rem}\label{rem-generalmartingale}
The martingale construction above applies more generally to branching diffusions with spatially dependent branching mechanism. Let us give a brief sketch as we will make use of it in Section \ref{sec-proofprobconstant} (see again   \cite{hardyharris} for a full account). Suppose we have a branching diffusion $Y$ on $[0,K]$ with (spatially dependent) branching mechanism $F(s,x)$, $s\in [0,1]$, $x\in(0,K)$ and set $F'(x,1):=\frac{d}{ds}F(x,s)|_{s=1}$. Let $\hat{\Upsilon}=(\hat{\Upsilon}(t),t\geq 0)$ be a unit-mean martingale for the single particle motion and accordingly, for $u\in N_t$, define $\hat{\Upsilon}_u(t)$ as the same object but with the associated single particle  replaced by the particle position $y_u(t)$. Then 
\be\label{eq-additivemartingalegeneral}
\hat{Z}(t)= \sum_{u\in N_t} \exp\left\{-\int_0^t F'(y_u(s),1) \ ds\right\} \hat{\Upsilon}_u(t)
\e
defines a martingale with respect to $\sigma(Y_t, t\geq 0)$. Changing measure with the martingale  $\hat{Z}$ induces a spine decomposition in which the spine has martingale density $\hat{\Upsilon}(t)$ with respect to the law of the single particle motion in $Y$. 
\end{rem}
Let us continue with the study of the martingale $Z^K$. Since we assumed ${E}(A\log^+ A)<\infty$, the following Proposition gives a necessary and sufficient condition for the $L^1(P^K_x)$-convergence of $Z^K$. 
\begin{prop}\label{theo-l1convergence} Recall that $\lambda(K) = (m-1)\beta-\mu^2/2-\pi^2/2K^2$ and let  $0<x<K$. \\
(i) If $\lambda(K) >0$ then the martingale $Z^K$ is  $L^1(P_x^K)$-convergent and in particular uniformly integrable. \\
(ii) If $\lambda(K) \leq 0$ then $\lim_{t\to\infty} Z^K(t)=0$ $P_x^K$-a.s.
\end{prop}
 We refrain from giving the proof of Proposition \ref{theo-l1convergence} since it is a straightforward adaptation of the proof of Theorem 13 in Kyprianou \cite{kyprianou} which presents the $L^1$-convergence result in the case of a branching Brownian motion with absorption at a space-time barrier, see also proof of Theorem 1 therein, as well as the proof in \cite{lyons} and the proof of Theorem A in \cite{lyonsetal}. \\ 
We will now show that the martingale limit $Z^K(\infty)$ is zero if and only if the process becomes extinct.
\begin{prop}\label{prop-zerolimitisextinction}
For $x\in(0,K)$, the events $\{Z^K(\infty)=0\}$ and $\{\zeta<\infty\}$ agree $P_x^K$-a.s.
\end{prop}
An essential idea in the proof of Proposition \ref{prop-zerolimitisextinction}  is to embed the  killed branching diffusion in a branching diffusion with killing on a larger strip. Let us introduce this procedure and some notation now as it will be used again in Section \ref{sec-rigorous}. \\  
Denote by $P_x^{(a,b)}$ the law under which $X$ is our usual branching Brownian motion but with killing upon exiting the interval $(a,b)$, where $- \infty \leq a <   b \leq \infty$ (and we simply write $P^b_x$ instead of $P_x^{(0,b)}$ in accordance with our previous notation). We denote by $\tau_u$ and $\sigma_u$ the birth respectively death time of a particle $u$ and write $v\leq u$ if $v$ is an ancestor of $u$ ($u$ is considered to be an ancestor of itself), in accordance with the classical Ulam-Harris notation (see for instance \cite{hardyharris}, p.290).  For an $\epsilon>0$, we choose $a$ and $b$ such that $a\leq 0 < \epsilon  \leq b$. Under $P_x^{(a,b)}$, we define
\bex
 \left. N_t\right|_{(0,\epsilon)}   = \left\{u\in N_t : x_v(s) \in(0,\epsilon) \ \forall v\leq u, \tau_v \leq s\leq \sigma_v\wedge t  \right\},
\ex 
which is the set of  particles $u\in N_t$ whose ancestors (not forgetting $u$ itself) have not exited  $(0,\epsilon)$ up to time $t$. Now we can define the restriction of $X$ to $(0,\epsilon)$ under $P_x^{(a,b)}$ by 
\bex
X_t|_{(0,\epsilon)}=\sum_{u \in N_t|_{(0,\epsilon)}} \delta_{x_u(t)}, \ \ \ t\geq 0.
\ex 
Then we conclude immediately that, for an initial position in $(0,\epsilon)$, the restricted process $\left.X\right|_{(0,\epsilon)}= \left(\left.X_t\right|_{(0,\epsilon)}, t\geq 0\right)$ under $P_x^{(a,b)}$ has the same law as $(X,P^\epsilon_x)$.
\begin{proof}[Proof of Proposition \ref{prop-zerolimitisextinction}]
Clearly $\{\zeta<\infty\}\subset \{Z^K(\infty)=0\}$ and it remains to show that survival implies that $Z(\infty)$ is strictly positive. We consider the cases $\lambda(K)\leq 0$ and $\lambda(K)>0$ separately. \\
Assume $\lambda(K)\leq 0$. 
Suppose for a contradiction that $Z^K(\infty)=0$ on survival. This requires the terms 
$e^{\mu x_u(t)} \sin(\pi x_u(t) /K)$ to vanish which can only happen if all particles move towards the killing boundaries $0$ and $K$. That is to say, for any $\epsilon>0$, all particles leave the interval $(\epsilon,K-\epsilon)$ eventually, and thus we may assume without loss of generality that the process survives in a small strip $(0,\epsilon)$, for any $\epsilon>0$.  We will now lead this argument to a contradiction by showing that, for $\epsilon$ small enough, the  ${P}^\epsilon_x$-branching diffusion, $x\in(0,\epsilon)$, will become extinct a.s.\\
We embed the ${P}^\epsilon$-branching diffusion in a ${P}^{(-\delta,\epsilon+\delta)}$-branching diffusion  according to the previously described procedure. Now we choose $\delta$ and $\epsilon$ small enough such that $\lambda(\epsilon+2\delta):= (m-1)\beta-\mu^2/2-\pi^2/2(\epsilon+2\delta)^2<0$. Then, under  ${P}^{(-\delta,\epsilon+\delta)}$, the process
\bex
 && Z^{(-\delta,\epsilon+\delta)}(t) \\
 && \quad :=\sum_{u\in N_t}  \left\{e^{\mu (x_u(t)+\delta)-\lambda(\epsilon+2\delta)t}
 \sin(\pi (x_u(t)+\delta) /(\epsilon+2 \delta))\right\}, \ \  t\geq 0,
\ex
is a martingale of the form in Proposition \ref{theo-l1convergence}.
Considering now the contribution coming from the particles in the set $N_t|_{(0,\epsilon)}$ only, we first note that survival of the  ${P}^\epsilon$-branching diffusion ensures that this set is non-empty for any time $t$. Further, for particles $u\in N_t|_{(0,\epsilon)}$, the terms  $e^{\mu (x_u(t)+\delta)} \sin(\pi (x_u(t)+\delta) /(\epsilon+2\delta))$ are uniformly bounded from below by a  constant $c>0$ and hence, under $P^{(-\delta,\epsilon+\delta)}_x$, we get
\bex
Z^{(-\delta,\epsilon+\delta)}(t)\geq c N_t|_{(0,\epsilon)} e^{-\lambda(\epsilon+2\delta)t}. 
\ex
Since we have chosen $\delta$ and $\epsilon$ such that $\lambda(\epsilon+2\delta)<0$, we now conclude that $Z^{(-\delta,\epsilon+\delta)}(\infty)=\infty$,  $P^{(-\delta,\epsilon+\delta)}_x$-a.s.  
This is a contradiction since $Z^{(-\delta,\epsilon+\delta)}$ is a positive martingale and therefore has a finite limit. Hence, for $\lambda(K)\leq 0$, the martingale limit $Z^K(\infty)$ cannot be zero on survival.\\
Assume now $\lambda(K)>0$. Suppose for a contradiction that $\{\zeta=\infty\}\cap\{Z^K(\infty)=0\}$ is non empty and work on this event from now on.  Now let $z_K(x)=P_x^K(Z^K(\infty)=0)$, for $x\in (0,K)$. 
Define $M_\infty:=\mbf{1}_{\{Z^K(\infty)=0\}}$ and set 
\bex
M_t:= E^K_x(M_\infty|\mcal{F}_t)=\prod_{u\in N_t}z_K(x_u(t)),
\ex
where the second equality follows from the branching Markov property. Then $(M_t, t \geq 0)$ is a uniformly integrable $P^K_x$-martingale with limit $M_\infty$. Hence its limit on the event  $\{\zeta=\infty\}\cap\{Z^K(\infty)=0\}$  is $1$, $P^K_x$-a.s. This requires in turn that all particles $x_u(t), u\in N_t$ move towards $0$ and $K$ as $t\to\infty$, since we know from Proposition \ref{theo-l1convergence} (i) that $z_K(x)<1$ for $x$ within $(0,K)$.  The previous part of this proof already showed that this leads to a contradiction. 
Thus, for $\lambda(K)>0$, the martingale limit cannot be zero on survival. This completes the proof. 
\end{proof}

\begin{proof}[Proof of Theorem \ref{theo-criteriaforpossurvivalprob}]
Note that $\lambda(K)\geq 0$ if and only if  $\mu<\sqrt{2(m-1)\beta}$ and $K>K_0$. The result follows now immediately from Proposition \ref{theo-l1convergence} and Propositon \ref{prop-zerolimitisextinction}.
\end{proof}

\begin{rem}\label{rem-fkpp} 
We can apply the  same argument given for the process $M$ in the proof of Proposition \ref{prop-zerolimitisextinction} to show that
\bex
E^K_x(\mbf{1}_{\{\zeta^K < \infty\}} | \mcal{F}_t) = \prod_{u\in N_t} \left( 1-p_K(x_u(t))\right), \ t\geq 0
\ex
is a uniformly integrable product martingale which, followed by a classical Feynman-Kac argument (cf. Champneys et al. \cite{champneysetal}), gives that $1-p_K(x)$  solves
\be\label{fkpptw}
L u+F(u) &=& 0\quad \mbox{on} \,\, (0,K) \nonumber\\
u(0)= u(K) &=& 1.
\e 
We will show later in Remark \ref{rem-uniqueness} at the end of Section \ref{sec-backbone} that, if there exists a non trivial solution to (\ref{fkpptw}), then it is unique. This  will then again imply  that  $\{Z^K(\infty)=0\}$ and $\{\zeta<\infty\}$  agree $P_x^K$-a.s.
\end{rem}


\section{Backbone decomposition via martingale changes of measure}\label{sec-backbone}
In this section we decompose the  $P^K$-branching diffusion into branching diffusions corresponding to the blue and red trees described in our intuitive picture in Section \ref{sec-introduction}. The blue tree which consists of all genealogical lines of descent that will never become extinct will be shown to correspond to a branching diffusion which we will henceforth refer to as the {\itshape{backbone}}. Secondly, the red trees which contain all remaining lines of descent will be shown to correspond to copies of the $P^K$-branching diffusion conditioned on becoming extinct. \\
The law of the branching diffusion corresponding to the coloured tree, $\mbf{P}^K$, is defined by the law of $X$ under $P^K$ and a subsequent deterministic colouring of the particles as described previously. Its natural filtration is $\tilde{\mcal{F}}_t:= \sigma\{\mcal{F}_t, c(u)_{u\in N_t} \}$ where $c(u)$ is the colour of a particle $u\in N_t$. We say a particle $u$ is blue if it has an infinite genealogical line of descent and we write $c(u)=b$, otherwise we say it is red and write $c(u)=r$.
We have, for all $t \geq 0$,
\bex
\left. \frac{d\mbf{P}^K_x}{dP^K_x}\right|_{\mcal{F}_\infty} = \prod_{u\in N_t} \Big( \mbf{1}_{\{c(u)=b\}}+ \mbf{1}_{\{c(u)=r\}}\Big)=1
\ex
and thus
\be\label{eq-changeofmeasureheur}
\left. \frac{d\mbf{P}^K_x}{dP^K_x}\right|_{\mcal{F}_t} &=& E^K_x\Big( \prod_{u\in N_t} \Big( \left. \mbf{1}_{\{c(u)=b\}}+ \mbf{1}_{\{c(u)=r\}}\Big) \right|\mcal{F}_t \Big) \nonumber\\
&=& \sum_{(c_u)_{u\in N_t}} \prod_{u\in N_t}   P^K_x \Big( c(u)=c_u \Big|\mcal{F}_t \Big) \nonumber\\
&=& \sum_{(c_u)_{u\in N_t}} \prod_{u\in N_t, c_u=b} p_K(x_u(t))     \prod_{u\in N_t, c_u=r}  \Big(1-p_K(x_u(t)) \Big)=1,
\e
where $(c_u)_{u\in N_t}$ is the set of all possible colourings of $N_t$. In particular, for $A\in\mcal{F}_t$,  we get
\bex
&& \mbf{P}^K_x(A; c(u)=c_u \ \forall u \in N_t | \mcal{F}_t) \\
&&\qquad\qquad = \mbf{1}_A   \prod_{u\in N_t, c_u=b} p_K(x_u(t))     \prod_{u\in N_t, c_u=r}  \Big(1-p_K(x_u(t)) \Big).
\ex
We can now easily derive the change of measure for the branching diffusion corresponding to the red tree. Let $A\in \mcal{F}_t$, and write $c(0)=r$ for the event that the initial particle is red and thus
\be\label{changeofmeasurepr}
\mbf{P}^{R,K}_x (A):= \mbf{P}_x^K(A | c(0)=r) &=& \frac{\mbf{P}_x^K(A ; c(u)=r \ \forall u\in N_t)}{\mbf{P}_x^K( c(0)=r )}  \nonumber \\
&=& \frac{E^K_x\Big(\mbf{1}_A \prod_{u\in N_t} \big(1-p_K(X_u(t))\big) \  \Big)}{1-p_K(x)}.
\e
Clearly, conditioning the genealogical line of the initial particle to become extinct agrees with conditioning the whole process to become extinct and therefore the law of $X$ under $\mbf{P}^{R,K}$ agrees with the law of $X$ conditioned on extinction. The following Proposition characterises the process under $\mbf{P}^{R,K}$.
\begin{prop}\label{prop-redtreeisextinction} 
For $\nu\in\mcal{M}_a(0,K)$, define  $\mbf{P}^{R,K}_\nu$ via  (\ref{changeofmeasurepr}).
Then $(X,\mbf{P}^{R,K}_\nu)$ is a branching process 
with single particle motion characterised by the infinitesimal generator
\bex
L^R_K  = \frac{1}{2} \frac{d^2}{dx^2} - \left(\mu + \frac{p_K'}{1-p_K} \right) \frac{d}{dx}\ \ \text{ on } \ (0,K),
\ex 
 for $u\in C^2(0,K)$ with $u(0+)=u(K-)=0$, 
and the branching activity is governed by the space-dependent branching mechanism
\bex
F_K^R(s,y)=\frac{1}{1-p_K(y)} (F(s(1-p_K(y)))-sF(1-p_K(y))), 
\ex   
for $s\in[0,1]$ and $y\in(0,K)$.

\end{prop}

\begin{proof}
The change of measure in  (\ref{changeofmeasurepr}) preserves the branching property  in the following sense. Let $\nu=\sum_{i=1}^n \delta_{x_i}$ be an initial configuration in $(0,K)$ and $A\in\mcal{F}_t$. Then 
\bex
\mbf{P}^{R,K}_\nu
(A) &=& E^{K}_\nu
\Big( \mbf{1}_A \frac{\prod_{u\in N_t} \big(1-p_K(x_u(t))\big)}{\prod_{i=1}^n \big(1-p_K(x_i)\big)} \Big) \\
&=&  \prod_{i=1}^n E^{K}_{{x_i}} \Big( \mbf{1}_A \frac{\prod_{u\in N_t} \big(1-p_K(x_u(t))\big)}{1-p_K(x_i)} \Big) \\
&=& \Big( \otimes_{i=1}^n \mbf{P}^{R,K}_{x_i} \Big) (A).
\ex
The process $(X,\mbf{P}^{R,K})$  is therefore completely characterised by its evolution up to  the first branching time $T$. Let $\xi = \{\xi_t , 0\leq t\leq T\}$ denote the path of the initial particle up to time $T$ and let $H$ be a positive bounded measurable functional of this path. We begin with considering the case $t<T$.  We have
\be\label{eq-reduptotau}
\mbf{E}^{R,K}_x(H(\xi_s,s\leq t) ; T>t) &=& \mbb{E}_x^{K} \Big(H(\xi_s,s\leq t) \frac{1-p_K(\xi_t)}{1-p_K(x)}; T>t \Big) \nonumber \\
&=& e^{-\beta t} \mbb{E}_x^{R,K} \Big(H(\xi_s,s\leq t) e^{- \int_0^t \frac{F(1-p_K(\xi_s))}{1-p_K(\xi_s)}  ds} \Big),  
\e
where $\mbb{P}^{R,K}_x$ is defined by the change of measure
\be\label{eq-redmotion}
\left.\frac{d\mbb{P}^{R,K}_x}{d\mbb{P}^K_x}\right|_{\mcal{G}_t} = \frac{1-p_K(\xi_t)}{1-p_K(x)} e^{\int_0^t \frac{F(1-p_K(\xi_s))}{1-p_K(\xi_s)}\ ds}, \ \ t\geq 0,
\e
and  $(\mcal{G}_t,t\geq 0)$ denotes the natural filtration of $(\xi,\mbb{P}_x^K)$. Thus the initial particle performs a $\mbb{P}^{R,K}$-motion  which is governed by the infinitesimal generator  $L^R_K$ as given in the statement of the proposition. Taking $H=1$ above, we see immediately that under $\mbf{P}^{R,K}$ the branching rate changes to
\be\label{eq-betar}
 \beta^R(y) &=& \frac{F(1-p_K(y))+ \beta(1-p_K(y))}{1-p_K(y)} = \beta \sum_{k\geq 0} q_{k} (1-p_K(y))^{k-1}, \nonumber\\
\e
for $ y\in(0,K)$. \\
It remains to identify the offspring distribution and we therefore study the process at its first branching time $T$.
We get
\be 
&&\mbf{E}^{R,K}_x(H(\xi_s, s\leq T); N_T=k; T\in dt) \nonumber\\
&=& E_x^{K} \Big( \frac{\big(1-p_K(\xi_T)\big)^{N_T}}{1-p_K(x)} H(\xi_s, s\leq T); T \in dt; N_T=k \Big) \nonumber \\
&=& E_x^{K} \Big( \frac{\big(1-p_K(\xi_T)\big)^{k}}{1-p_K(x)} H(\xi_s, s\leq T) \beta e^{-\beta T}  q_k \Big) \nonumber\\
&=& \mbb{E}^{R,K}_x \Big( H(\xi_s, s\leq T) \beta^R(\xi_T) e^{- \int_0^T \beta^R(\xi_s)  ds} \ \frac{\beta}{\beta^R(\xi_T)} {q_k (1-p_K(\xi_T))^{k-1}} \Big). \nonumber
\e
We see that, in addition to the change in the motion and the branching rate, the offspring distribution under $\mbf{P}^{R,K}$ becomes $\{q_k^R, k\geq 0\}$ where
\be\label{eq-qr}
q^R_k(y) &=& \beta (\beta^R(y))^{-1} q_{k} (1-p_K(y))^{k-1} , \ \ k\geq 0.
\e
A simple computation shows that $F_K^R(s,y)= \beta^R(y) (\sum_{k\geq 0} q_k^R(y) s^{k} - s)$ takes the desired form.
\end{proof}
Similarly to the above reasoning, we obtain the law of the dressed backbone, that is a backbone  with  immigration of $\mbf{P}^{R,K}$-branching diffusions, by conditioning on the first particle being blue. Thus
\be\label{eq-changemeasurepd}
\mbf{P}_x^{D,K}(A)&:=& \mbf{P}_x^K(A | c(0)=b) \nonumber\\
&=& \frac{\mbf{P}_x^K(A ; c(u)=b  \text{ for at least one } u\in N_t)}{\mbf{P}_x^K( c(0)=b )}  \nonumber\\
&=& \frac{E^K_x\Big(\mbf{1}_A \Big( 1- \prod_{u\in N_t}  \big(1-p_K(x_u(t))\big) \Big)\  \Big)}{p_K(x)}.
\e
Then $(X,\mbf{P}^{D,K})$ certainly agrees  with $(X,P^K)$ conditioned on survival. Let us  characterise the evolution under $\mbf{P}_x^{D,K}$. We use the previous notation and in addition let $\tau=T\wedge \tau_{(0,K)}$ denote the death time of the initial particle, where $\tau_{(0,K)}$ is the first time this particle exits $(0,K)$. Then
\be\label{eq-dresseduptotaupre}
 \mbf{E}^{D,K}_x(H(\xi_s,s\leq t) ; \tau>t) &=& e^{-\beta t} \mbb{E}_x^{K} \Big(H(\xi_s,s\leq t) \ \frac{p_K(\xi_t)}{p_K(x)}, \tau_{(0,K)}>t \Big) \nonumber \\
&=& e^{-\beta t} \mbb{E}_x^{B,K} \Big(H(\xi_s,s\leq t) \ e^{\int_0^t \frac{F(1-p_K(\xi_s))}{p_K(\xi_s)} \ ds} \Big),  
\e
where $\mbb{P}^{B,K}_x$ is defined by the change of measure,  for $t\geq 0$, 
\be\label{eq-bluemotion}
\left.\frac{d\mbb{P}^{B,K}_x}{d\mbb{P}^K_x}\right|_{\mcal{G}_t} &=& \frac{p_K(\xi_t)}{p_K(x)} \exp\left\{ - \int_0^t  \frac{F(1-p_K(\xi_s))}{p_K(\xi_s)} \ ds \right\} \ \mbf{1}_{\{\tau_{(0,K)}>t\}}, \nonumber  \\
\e
and  $(\mcal{G}_t,t\geq 0)$ is again the natural filtration of $(\xi,\mbb{P}^K_x)$. 
Thus,  setting
\be\label{eq-branchingratedressed}
 \beta^D(x) &=&   - \frac{F(1-p_K(x)) - \beta p_K(x)}{p_K(x)} \nonumber \\
 &=& \beta \frac{1-\sum_{k=0}^{\infty} (1-p_K(x))^{k} q_{k}}{p_K(x)}, \ \text{ for } x\in(0,K), 
\e
we see that (\ref{eq-dresseduptotaupre}) simplifies to
\be\label{eq-dresseduptotau}
 \mbf{E}^{D,K}_x(H(\xi_s,s\leq t) , \tau>t)  =  \mbb{E}_x^{B,K} \Big(H(\xi_s,s\leq t) e^{- \int_0^t \beta^D(\xi_s) \ ds} \Big).  \nonumber \\
\e
We deduce from this that, under $\mbf{P}^{D,K}$, the motion of the initial particle is given by the change of measure in (\ref{eq-bluemotion}) and it branches at rate $\beta^D(\cdot)$ as in (\ref{eq-branchingratedressed}). \\
It remains to specify the offspring distribution. We begin with the expression in (\ref{eq-changemeasurepd}) and then use  (\ref{eq-bluemotion}) and (\ref{eq-branchingratedressed})  to get
\bex
&& \mbf{E}^{D,K}_x(H(\xi_s, s\leq T); T\in dt; N_T=k) \\ &=& E^{K}_x\Big( \frac{1-(1-p_K(\xi_T))^k}{p_K(x)} H(\xi_s, s\leq T) \beta e^{-\beta T} q_k \Big)\\
&=& \mbb{E}^{B,K}_x \Big( H(\xi_s, s\leq T) \beta^D e^{-\int_0^T \beta^D(\xi_s) ds} \frac{\beta}{\beta^D(\xi_T)} \ q_k \frac{1-(1-p_K(\xi_T))^{k}}{p_K(\xi_T)} \Big). 
\ex
Again this reveals the evolution of the initial particle as described above and we further see that the offspring distribution of the initial particle under $\mbf{P}^{D,K}$ is given by $\{q_k^D, k\geq 0\}$ where 
\bex
q_k^D(x) \propto  q_k \frac{1-(1-p_K(\xi_T))^k}{p_K(\xi_T)} , \ \text{ for }
 x\in (0,K)
\ex
up to the normalising constant $\beta (\beta^D(x))^{-1}$.
We note that $q_0(x)=0$ for all $x\in (0,K)$ which we expected to see since $(X,\mbf{P}^{D,K})$ is equal in law to $(X,P^K)$ conditioned on survival.
However, we have so far neglected the fact that the initial particle can give birth to particles of the same type, i.e. blue particles (referred to as branching), and red particles which evolve as under $\mbf{P}^{R,K}$(referred to as immigration). We will split up the rate $\beta^D$ and the offspring distribution $q_K^D$ into terms corresponding to branching respectively immigration. Firstly, note that we can decompose the rate $\beta^D$ into
\be\label{eq-branchingratesdecomp}
\beta^D(y) &=& \beta \sum_{k\geq 2}  \sum_{n\geq k} q_n \binom{n}{k} p_K(y)^{k-1}(1-p_K(y))^{n-k}  \nonumber \\
&& \quad + \beta \sum_{n\geq 1}  q_n n (1-p_K(y))^{n-1} \nonumber\\
&=:& \beta^B(y) + \beta^I(y).
\e
Then $\beta^I$ is the rate at which the initial particle gives birth to one blue particle and a random number of (red) immigrants (immigration rate) while $\beta^B$ is the rate at which the initial particle  gives birth to at least two  particles of the blue type and a random number of (red) immigrants occur (branching rate of the branching diffusion corresponding to the blue tree). We can now rewrite the offspring distribution $q_k^D$ as 
\be
q_k^D & \propto & q_k \ 
\frac{1- (1-p_K(x))^{N_T}}{p_K(x)} \nonumber \\
 & = &  q_k \  \sum_{i=2}^{k} \binom{k}{i}  {p_K(x)^{i-1} (1-p_K(x))^{k-i}}  
 \label{eq-bluebluebranching}\\
&& \qquad\qquad +  q_k \    {(1-p_K(x))^{k-1}}, \ \ k\geq 1.
    \label{eq-blueredbranching} 
 \e
Then the term in (\ref{eq-bluebluebranching}) gives, up to normalisation, the probability that  the initial particle branches into $i$  particles of its type and, at the same branching time, $k-i$ particles immigrate. The term in (\ref{eq-blueredbranching}) is the the probability that  $k-1$   immigrants occur, again up to a normalising constant. \\
Note that $(X,\mbf{P}^{D,K})$ inherits the branching Markov property from $(X,P^K)$ by (\ref{eq-changemeasurepd}) in a similar spirit to the case of $(X,\mbf{P}^{R,K})$ (cf. the proof of Proposition \ref{prop-redtreeisextinction}).  Thus the  description of the initial particle also  characterises the evolution of all particles of the blue type and together with the characterisation of the immigrating $\mbf{P}^{R,K}$-branching diffusions in Proposition \ref{prop-redtreeisextinction} we have completely characterised the evolution of  $X$ under $\mbf{P}^{D,K}$. 
The following result is now an immediate consequence.

\begin{theo}[The dressed Backbone]\label{theo-dressedbackbone}
Let $K>K_0$ and $x\in(0,K)$. The process $(X,\mbf{P}^{D,K}_x)$ evolves as follows. 
\begin{enumerate}
\item From $x$, we run a $\mbb{P}^{B,K}_{x}$-diffusion  which dies at rate $\beta^B$,
\item at the space-time position of its death, it is replaced by  $A^B$ particles where $A^B$ is distributed according to the probabilities, 
\be\label{eq-offspringdistributionblue}
q^B_k(y) = \beta \ \beta^B(y)^{-1} { 
\sum_{n\geq k} q_n \binom{n}{k} p_K(y)^{k-1} (1-p_K(y))^{n-k}}, \nonumber \\
\e 
for  $ k\geq 2$ and $y\in(0,K)$.
\item Each of the offspring particles repeats its parent's stochastic behaviour. 
\item Conditionally on the branching diffusion, say $X^B$,  generated by steps (i) - (iii), we have the following.
\begin{itemize}
\item (Continuous immigration) Along the trajectories of each particle in $X^B$, an immigration with $n\geq 1$ immigrants occurs at rate 
 \be\label{eq-continuousimmigration}
 \beta^I_n(y) = \beta q_{n+1} (n+1) (1-p_K(y))^{n}, \ \ y\in(0,K). \nonumber \\
 \e
\item (Branch point immigration) At a branch point of $X^B$ with $k\geq 2$  particles, we see an immigration of $n\geq 0$ immigrants with probability 
\be\label{eq-branchpointimmigration}
q_n^I(y) = q_{n+k} \binom{n+k}{k} p_K(y)^{k-1} (1-p_K(y))^{n}, \ \ y\in(0,K). \nonumber \\
\e
\end{itemize} 
Each immigrant initiates an independent copy of $(X,\mbf{P}^{R,K})$ from the space-time position of its birth. \end{enumerate}
\end{theo}
The following Theorem is a more precise version of Theorem \ref{theo-backboneintro},  stated in Section \ref{sec-introduction}.
\begin{theo}[Backbone decomposition]\label{theo-decomposition}\label{cor-backbonedecomposition}
Let $K>K_0$ and $\nu\in\mcal{M}_a(0,K)$ such that $\nu= \sum_{i=1}^n \delta_{x_i}$ with $x_i \in (0,K)$, $n\geq 1$. For $t \geq 0$, we can define
\be\label{eq-changeofmeasurebackbone}
 \left.\frac{d \mbf{P}^K_\nu}{d P^K_\nu}\right|_{\mcal{F}_t}  
&=& \sum_{k=0}^n \sum_{(x_1,...,x_k)} \prod_{i=1}^k p(x_i) \left. \frac{d\mbf{P}^{D,K}_{x_i}}{dP^K_{x_i}}\right|_{\mcal{F}_t}  \prod_{j=k+1}^n (1-p(x_j)) \left.\frac{d\mbf{P}^{R,K}_{x_j}}{d P_{x_j}^K}\right|_{\mcal{F}_t},  \nonumber \\
\e
where the second sum above is taken over all $k$-tuples of $x_1,...,x_n$. However note that the right-hand side of (\ref{eq-changeofmeasurebackbone}) is equal to $1$ and thus,  trivially, on the filtration $(\mcal{F}_t)_{t\geq 0}$, $(X,\mbf{P}^K_\nu)$ is Markovian and  equal in law to $(X,P^K_\nu)$.

\end{theo}

\begin{proof}  The change of measure is just a restatement of (\ref{eq-changeofmeasureheur}) and the result follows from  Proposition \ref{prop-redtreeisextinction} and Theorem \ref{theo-dressedbackbone}.
\end{proof}

Intuitively speaking, we can describe the evolution under $\mbf{P}^{D,K}_{\nu}$ and thus also under $P^K_\nu$ as follows. Independently for each initial particle $x_i$, we flip a coin with probability $p(x_i)$ of `heads'. If it lands `heads', we initiate a copy of $(X,\mbf{P}^{D,K}_{x_i})$ and otherwise we initiate a copy of $(X,\mbf{P}^{R,K}_{x_i})$.  

\begin{cor}\label{rem-thinning} 
Given the number of particles of $({X},{P}^K_\nu)$ and their positions, say $x_1,...,x_n$ for some $n\in \mbb{N}$, at a fixed time $t$, the number of particles of $X^B_t$ is the number of successes in a sequence of $n$ independent Bernoulli trials each with success probability $p_K(x_1),...,p_K(x_n)$. 
\end{cor}

We refer to the branching diffusion generated by steps (i)-(iii) of Theorem \ref{theo-decomposition} above as the backbone. Its law can be characterised as follows. 
\begin{prop}[The backbone]\label{defi-bluetree}
For $\nu\in\mcal{M}_a(0,K)$  such that $\nu= \sum_{i=1}^n \delta_{x_i}$ with $x_i \in (0,K)$, $n\geq 1$, we define the measure $\mbf{P}^{B,K}_\nu$ via the following change of measure. For $t\geq 0$, 
\bex
\left.\frac{d\mbf{P}^{B,K}_\nu}{d{P}_\nu^K}\right|_{\mcal{F}_t} 
&=& \prod_{v\in \mcal{T}_t} \frac{p_K(x_v(\sigma_v\wedge t))}{p_K(x_v(\tau_v))} \mbf{1}_{\{t<\tau^v_{(0,K)}\}}\ \\
&& \times  \exp\left\{ \int_{\tau_v}^{\sigma_v \wedge t} F'(1-p_K(x_v(s))) + \beta \ ds \right\} \\
&& \times \prod_{v\in \mcal{T}_{t-}}  \frac{q^B_{A_v}(x_v(\sigma_v))}{q_{A_v} \beta(x_v(\sigma_v)) (\beta^B(x_v(\sigma_v)))^{-1}}, \ \   
\ex
where $\mcal{T}_t$ is the set of all particles $v\in\mcal{T}$ with $\tau_v<t$  and $v$ is in $\mcal{T}_{t-}$ if, in addition, $\sigma_v<t$ . As usual, $\tau_v$ and $\sigma_v$ are the birth respectively death times, $\tau^v_{(0,K)}$ is the first exit time from $(0,K)$ and $A_v$ is the random number of offspring of a particle $v\in\mcal{T}_{t-}$. \\
The branching diffusion $(X,\mbf{P}^{B,K}_\nu)$ has infinitesimal generator 
\bex
L^{B}_K = \frac{1}{2}\frac{d^2}{dx^2} - \left(\mu - \frac{p_K'}{p_K}\right) \frac{d}{dx}\ \ \text{ on } \ (0,K),
\ex 
defined  for all $u\in C^2(0,K)$, and space-dependent branching mechanism
\bex
F_K^B(s,y)= \frac{1}{p_K(y)} \left( F(s p_K(y)+(1-p_K(y)))- (1-s) F(1-p_K(y)) \right),
\ex 
for  $ s\in [0,1]$ and $y\in(0,K)$.
The process $(X,\mbf{P}^{B,K}_x)$ evolves according to the steps (i)-(iii) of Theorem \ref{theo-decomposition}.
\end{prop} 
\begin{proof}
First note that the motion under $\mbb{P}^{B,K}$, given by the change of measure in (\ref{eq-bluemotion}), is governed by the infinitesimal generator $L_K^B$ as given in the statement. A simple computation also shows  that $F^{B}_K(s,y) = \beta^B(y) (\sum_{k\geq 2} q_k^B(x) s^{k} -s )$ with $\beta^B$ and $q_k^B$ as in  (\ref{eq-branchingratesdecomp}) and (\ref{eq-offspringdistributionblue}) gives the desired form.
The result then follows from rewriting the change of measure up to the first branching time $T$ as
\bex
\left.\frac{d\mbf{P}^{B,K}_x}{d{P}^K_x}\right|_{\mcal{F}_T} &=& \frac{p_K(\xi_T)}{p_K(x)}  \exp\left\{ -  \int_{0}^{T} \frac{F(1-p_K(\xi_s))}{p_K(\xi_s)} \ ds \right\} \mbf{1}_{\{t<\tau_{(0,K)}\}}\\
&& \times \frac{1}{\beta} \beta^{B}(\xi_T))  \exp\left\{ -  \int_{0}^{T} \beta^B(\xi_s)-\beta \ ds \right\}  
\times 
\frac {q^B_{N_T}(\xi_T))} {q_{N_T}}, 
\ex
noting that the first line on the right-hand side accounts for the change of motion, the first term in the second line for the change in the branching rate and the last  term in the second line for the change in the offspring distribution. 
\end{proof}

\begin{rem}\label{rem-uniqueness}
As promised earlier, with the help of Corollary \ref{rem-thinning}, we can show that, if  (\ref{fkpptw}) has a non-trivial solution, then it is unique. Assume $g_K(x)$ is a non-trivial solution to (\ref{fkpptw}). By a Feynman-Kac argument (cf. Champneys et al. \cite{champneysetal}), it follows that
\bex
M^K(t) = \prod_{u\in N_t}  g_K(x_u(t)), \ t\geq 0,
\ex
is a $P_x^K$-product martingale. Since $M^K$ is  uniformly integrable,  its limit $M^K(\infty)$ exits $P^K_x$-a.s. On the event of extinction, $M^K(\infty)=1$. On the event of survival, we have
\be\label{eq-productmartingalelimit}
M^K(t) = \prod_{u\in N_t} g_K(x_u(t)) \leq \prod_{u\in N_t^B} g_K(x^B_u(t)),
\e
where $N_t^B$ is the set and $x_u(t)$ are the spatial positions of the particles in $X^B_t$. Clearly $|N_t|^B\to \infty$ as $t\to\infty$ since each particle in $X^B$ is replaced by at least two offspring  and there is no killing. Further particles in $X^B$ perform an ergodic motion and it is therefore not possible that $\liminf  g(x_u(t))$ tends to $1$.  Thus the right-hand side of (\ref{eq-productmartingalelimit}) tends to $0$ and we conclude that $M^K(\infty)= \mbf{1}_{\{\zeta <\infty\}}$. Hence $g_K(x)=E_x^K(M^K(\infty))=P_x^K(\zeta<\infty)$ which implies uniqueness. \\
In particular we may conclude that (\ref{fkpptw}) has a non-trivial solution if and only if $\mu<\sqrt{2(m-1)\beta}$ and $K>K_0$. 
\end{rem}

\section{Proof of Theorem \ref{theo-survivalprobrough}}\label{sec-survivalprobpart1}
We break up Theorem \ref{theo-survivalprobsonstant} into two parts which will be proved in the subsequent  sections.
\begin{prop}\label{prop-survivalprobasymptotics} 
Uniformly for all $x\in(0,K_0)$,
\bex
p_K(x) \sim c_K \sin(\pi x /K_0) e^{\mu x}, \ \ \ \ \text{as} \  K\downarrow K_0,
\ex
where $c_K$ is independent of $x$ and $c_K\downarrow 0$ as $K\downarrow K_0$. \\
\end{prop}
\begin{prop}\label{prop-constant}
The constant $c_K$ in Proposition \ref{prop-survivalprobasymptotics} satisfies
\be\label{eq-probconstant}
c_K 
 &\sim& (K-K_0)  \frac{(K_0^2\mu^2+\pi^2)(K_0^2 \mu^2 + 9 \pi^2)}{12 (m-1)\beta \pi K_0^{3}  (e^{\mu K_0}+1)}
 \ \ \ \ \text{as} \ \ K\downarrow K_0.
\e
\end{prop}
Theorem \ref{theo-survivalprobsonstant} then follows by defining ${C}_K$ to be the expression on the left-hand side in (\ref{eq-probconstant}).  \\
We will provide entirely probabilistic proofs of the results above.  We remark that, although it would take some effort to make rigorous, it is also possible to recover the asymptotics  of $p_K$ and the explicit  constant ${C}_K$ in an analytic approach using a careful asymptotic expansion of the non-linear ODE $Lu+F(u)=0$ with $u(0)=u(K)=1$, as shown to us by B. Derrida. \\  

\subsection{Proof of Proposition \ref{prop-survivalprobasymptotics}}\label{sec-proofroughprob}
We begin with a preliminary  result which  ensures that the survival probability $p_K$ is right-continuous at $K_0$.
\begin{lem}\label{prop-survivalcts}
Let $x \in (0,K_0)$. 
Then $\lim_{K\downarrow K_0} p_K(x) = 0$.
\end{lem}

\begin{proof}
We fix $x\in (0,K_0)$ throughout the proof and  consider $p_K(x)$ as a function in $K$.  For a fixed  $t>0$, let us  define the probability $p_K(x,t):=P^K_x(\text{survival in } (0,K) \text{ up to time } t)$. By monotonicity of measures we have  $\lim_{K\downarrow K_0} p_{K}(x,t)=p_{K_0}(x,t)$.  
Now we can write $p_{K}(x)= \inf_{t>0} p_K(x,t)$.  Hence $p_K(x)$ is  the infimum of a sequence of functions which are continuous at $K_0$ and thus upper semicontinuous at $K_0$, that is
\bex
\limsup_{K\downarrow K_0} p_K(x) \leq p_{K_0}(x).
\ex
Furthermore, $p_K(x)$ is decreasing as $K\downarrow K_0$ and bounded, so the right limit exists and 
\bex
p_{K_0}(x)\leq \lim_{K\downarrow K_0} p_K(x).
\ex
Combining the two inequalities above we obtain
right-continuity of $p_K(x)$ at $K_0$. By Theorem \ref{theo-criteriaforpossurvivalprob}, $p_{K_0}(x)=0$ and so we have $\lim_{K\downarrow K_0} p_K(x)=0$.

\end{proof}

The following lemma is the essential part in the proof of Proposition \ref{prop-survivalprobasymptotics}. Its proof is guided by the ideas in Aidekon and Harris \cite{aidekonharris}. 
\begin{lem}\label{lem-survivalsin} 
Let $y \in(0,K_0)$. 
Then we have
\be\label{eq-asymptoticsfraction}
\lim_{K\downarrow K_0}\frac{p_K(x)}{p_K(y)} = \frac{\sin(\pi x / K_0)}{\sin(\pi y /K_0)} e^{\mu(x-y)}.
\e
 uniformly for all $x\in(0,K_0)$. 
\end{lem}

\begin{proof} 
Fix $y\in(0,K_0)$. We begin with showing that the asymptotics hold  uniformly for all $x\in(0,y)$. \\
Let $\mcal{T}$ denote the set of  labels of particles realised in $(X,P_x^K)$. Define $T_{y}$ as the set containing all particles which are the first ones in their genealogical line to exit the strip $(0,y)$, i.e.
\bex
 T_{y} &=& \{u\in\mcal{T} : \exists s\in[\tau_u,\sigma_u] \text{ s.t. } x_u(s)\notin (0,y) \\
 && \qquad\qquad\qquad \text{ and } x_v(r)\in(0,y) \text{ for all } v<u, r\in[\tau_v,\sigma_v] \},
\ex
where $v<u$ means that $v$ is a strict ancestor of $u$. Further, for $u\in T_{y}$ denote by $T_{y}^u$ the first exit time of $u$  from $(0,y)$.  The random set $T_{y}$ is a stopping line in the sense of Biggins and Kyprianou \cite{bigginskyprianou} (see also Chauvin \cite{chauvin} which uses a slightly different terminology though). 
Since $y\in(0,K_0)$ the width of the strip $(0,y)$ is subcritical and hence, for any initial position $x\in(0,y)$, all particles will exit it eventually which ensures that $T_{y}$ is a dissecting stopping line. 
Now let $|{T_y}|$ be the number of particles which are the first ones in their line of descent to hit $y$, which can be written as
\bex
|{T_y}|=\sum_{u\in T_{y}} \mbf{1}_{\left\{x_u(T_{y}^u) = y\right\}}.
\ex
Recall from Remark \ref{rem-fkpp} that $(\prod_{u\in N_t} (1-p_K(x_u(t))), t \geq 0)$ is a $P_x^K$-martingale.  Since $T_{y}$ is dissecting it follows from \cite{chauvin} that we can stop the martingale  at $T_y$ 
and obtain, for $x\in(0,y)$,
\be\label{eq-stopmartingale}
1- p_K(x)= E_x^K\left(\prod_{u\in T_y} 1-p_K(x_u(T_{y}^u))\right) = E_x^K((1-p_K(y))^{|{T_y}|}), \nonumber \\
\e 
where we have used that the process started at zero becomes extinct immediately, i.e. $p_K(0)=0$. 
Further $|{T_y}|$ has the same distribution under $P_x^K$ and $P_x^{K_0}$ since we consider particles stopped at level $y$ below $K_0$ and thus we can replace $E_x^K$ by $E_x^{K_0}$ on the right-hand side above. Now, using first (\ref{eq-stopmartingale}) and then the geometric sum $\sum_{j=0}^{n-1} a^j =\frac{1-a^{n}}{1-a}$, we get
\be\label{eq-geometricsum}
\frac{p_K(x)}{p_K(y)}= E_x^{K_0} \left(\frac{1-\left(1-p_K(y)^{|{T_y}|}\right)}{1-(1-p_K(y))}\right) = E_x^{K_0}\left(\sum_{j=0}^{|{T_y}|-1} (1-p_K(y))^j\right). \nonumber \\
\e
The sum on the right-hand side is dominated by $|{T_y}|$ which does not depend on $K$ and has finite expectation (which will shortly be shown below). We can therefore apply the Dominated convergence theorem to the right-hand side in (\ref{eq-geometricsum}) and we conclude that 
\be\label{eq-domconv}
&&\lim_{K\downarrow K_0} E_x^{K_0}\left(\sum_{j=0}^{|{T_y}|-1} (1-p_K(y))^j\right) \nonumber\\
&&\qquad\qquad = E_x^{K_0}\left(\sum_{j=0}^{|{T_y}|-1} \lim_{K\downarrow K_0} (1-p_K(y))^j\right) =
E_x^{K_0}(|{T_y}|), 
\e
where the convergence holds point-wise in $x\in(0,y)$. In order to get uniform convergence we observe the following. We set $\varphi(x,K)= E_x^{K_0}\left(\sum_{j=0}^{|{T_y}|-1} (1-p_K(y))^j\right)$, for $x\in[0,y]$ (with the convention that the  $P^K$-branching diffusion becomes extinct immediately for initial position $x=0$ respectively stopped for $x=y$) and denote by $\varphi(x)=E_x^{K_0}(|{T_y}|)$ its point-wise limit. Since $1-p_K(y)\leq 1-p_{K'}(y)$, for $K\geq K'$, we have $\varphi(x,K)\leq \varphi(x,K')$ and thus, for any $x\in[0,y]$,  the sequence $\varphi(x,K)$  is monotone increasing as $K\downarrow K_0$. Moreover the functions $\varphi(x,K)$ and $\varphi(x)$ are continuous in $x$, for any $K$. In conclusion, we have an increasing sequence  of continuous functions on a compact set with a continuous point-wise limit and therefore  the convergence also holds uniformly in $x\in[0,y]$ (see e.g. \cite{rudin}, Theorem 7.13). \\
Combining (\ref{eq-geometricsum}) and (\ref{eq-domconv}) and the uniformity argument, we arrive at
\be\label{eq-auxeq}
\lim_{K\downarrow K_0} \frac{p_K(x)}{p_K(y)} = E_x^{K_0}(|{T_y}|),
\e
where, for fixed $y$, the convergence holds uniformly in $x\in(0,y)$.
Now let $\tau_\xi:=\inf\{t>0: \xi_t\in(0,y)\}$ be the first time a Brownian motion $\xi$ with drift $-\mu$ exists the interval $(0,y)$. Since $T_y$ is  dissecting  it follows from Theorem 6 in \cite{kyprianou} that we can apply the Many-to-one Lemma (see e.g. \cite{hardyharris} Theorem 8.5) for the stopping line $T_y$. This gives
\bex
&& E_x^{K_0}(|{T_y}|) \\
&=& {\mbb{Q}}_x^{K_0} \left(  \frac{\sin(\pi x / K_0)e^{\mu x}}{\sin(\pi \xi_{\tau_{\xi}}/K_0)   e^{\mu \xi_{\tau_{\xi}} + (\mu^2/2+\pi^2/2K_0^2)\tau_{\xi_y}}} e^{(m-1)\beta \tau_{\xi_y}},  \mbf{1}_{ (\xi_{\tau_{\xi_y}}=y)}\right) \\
&=&  \frac{\sin(\pi x / K_0)}{\sin(\pi y /K_0)} e^{\mu(x-y)}{\mbb{Q}}_x^{K_0}(\xi_{\tau_{\xi_y}}=y), 
\ex  
where we have used that $(m-1)\beta-\mu^2/2-\pi^2/2K_0^2=0$ (and  ${\mbb{Q}}_x^{K_0}$ is used as an expectation operator).
Under ${\mbb{Q}}_x^{K_0}$, $\xi$ will never hit $0$ since it is conditioned to stay in $(0,K_0)$. However as $\xi$ is positive recurrent it will eventually cross $y$ and therefore ${\mbb{Q}}_x^{K_0}(\xi_{\tau_{\xi_y}}=y)=1$. This proves our earlier claim that $|T_y|$ has finite expectation and  together with (\ref{eq-auxeq}) it completes the argument. \\
For uniformity for all  $x\in (0,K)$, it remains to show that (\ref{eq-asymptoticsfraction}) also holds uniformly for $x\in(y,K_0)$.
Instead of approaching criticality by taking the limit in $K$ we can now fix a $K>K_0$ and consider a (supercritical) strip $(z,K)$ and let $z\uparrow z_0 $ where $z_0:= K-K_0$.  
Denote by $p_{(z,K)}(x+z)$  the probability of survival in the strip $(z,K)$ when starting from $x+z$. 
We then have
\bex
\lim_{K\downarrow K_0} \frac{p_K(x)}{p_K(y)} = \lim_{z\uparrow z_0}\frac{p_{(z,K)}(x+z)}{p_{(z,K)}(y+z)}. 
\ex
Hence (\ref{eq-asymptoticsfraction}) is equivalent to showing that, uniformly for $x\in (y,K_0)$, 
\bex
\lim_{z\uparrow z_0}\frac{p_{(z,K)}(x+z)}{p_{(z,K)}(y+z)} = \frac{\sin (\pi x / K_0)}{\sin (\pi y /K_0)} e^{\mu (x-y)}. 
\ex
Then consider the stopping line containing all particles which exit the strip $(y+z,K)$ and  accordingly the set of particles which are the first in their genealogical line to exit $(y+z,K)$ at $y$.  Noting that the latter has the same law under $P_{x+z}^{z,K}$ and $P^{z_0,K}_{x+z}$, we can then repeat the argument in the first part.
\end{proof}

\begin{proof}[Proof of Proposition \ref{prop-survivalprobasymptotics}]
Choose a $y\in (0,K_0)$. Then an application of Lemma \ref{lem-survivalsin} gives, {as} $  K\downarrow K_0$,
\bex
p_K(x) = p_K(y) \frac{p_K(x)}{p_K(y)} \sim p_K(y) \frac{\sin(\pi x/K_0)}{\sin(\pi y / K_0)} e^{\mu (x-y)} = c_K \sin(\pi x /K_0) e^{\mu x}, 
\ex
uniformly for all $x\in(0,K_0)$,  where $c_K:=\frac{p_K(y)}{\sin(\pi y /K_0)} e^{-\mu y}$. By Proposition \ref{prop-survivalcts},  $c_K\downarrow 0$ as $K\downarrow K_0$ which completes the proof. 
\end{proof}

\subsection{Proof of Proposition \ref{prop-constant}}\label{sec-proofprobconstant}

In this section we will present the proof of Proposition \ref{prop-constant} which gives an explicit expression for the constant ${c}_K$ appearing in the asymptotics for the survival probability in Proposition \ref{prop-survivalprobasymptotics}. 
We outline  a heuristic derivation of the explicit constant $c_K$ in Proposition \ref{prop-constant}  which will give the intuition for the rigorous proofs presented subsequently.

\subsubsection{Heuristic argument}\label{sec-heuristics} We break up the heuristic argument for Proposition \ref{prop-constant} into four steps which we will refer back to in the rigorous proof in the subsequent section.\\
{\it Step (i) (The growth rate of the backbone)} Consider a process $Y^B=(Y^B_t, t\geq 0)$ performing the single particle motion of the backbone, that is according to the infinitesimal generator $L^B_K$ which is given in Proposition \ref{defi-bluetree} as
\bex
L^{B}_K = \frac{1}{2}\frac{d^2}{dx^2} - \left(\mu - \frac{p_K'}{p_K}\right) \frac{d}{dx}\ \ \text{ on } \ (0,K),
\ex
with domain $C^2(0,K)$.
 Let  $\Pi^B_K$ be the invariant density for $L^B_K$, i.e. the positive solution of $\tilde{L}_K^B \Pi_K^B = 0$ where $\tilde{L}_K^B$ is the formal adjoint of ${L}_K^B$. Then 
\bex \Pi^B_K(x) \propto p_K(x)^2 e^{-2\mu x}, \ \ \ x\in(0,K).
\ex 
For $t\geq 0$ , we define $\Gamma(t,A)=\int_0^t \mbf{1}_{\{Y^B_s\in A\}} ds$, $A\subset [0,K]$, to be the occupation time up to $t$ of $Y^B$ in $A$. Then  large deviation theory suggests that the probability that the measure $t^{-1} \Gamma(t,\cdot)$ is  `close' to $ \int_0^K \mbf{1}_{\{\cdot\}}(y) f^2(y) \Pi^B_K(y) \, dy$ should be roughly \bex
\exp\left\{t \int_0^K L^B_K f(y) \ f(y) \Pi^B_K(y) \ dy \right\}.
\ex 
Now, as each particle in the backbone moves according to $L_K^B$, we guess that the expected number of particles at time $t$ with occupation density like $f^2\Pi^B_K$  is very roughly 
\be\label{eq-largedeviationexpecnumber}
\exp\left\{t \int_0^K (L^B_K+ {F^B_K}'(1,y)) f(y) \ f(y) \Pi^B_K(y) \ dy \right\},
\e 
where 
\bex
{F_K^B}'(1,x):=\frac{d}{ds}F^B(s,x)|_{s=1} = (m-1)\beta+\frac{F(1-p_K(x))}{p_K(x)}, \ \ \  x\in(0,K).
\ex
The expected growth rate of the blue tree is given by maximising the integral appearing in    (\ref{eq-largedeviationexpecnumber}) over all $f$ with $\int_0^K f^2(x) \Pi_K^B(x) dx = 1$.
We can compute this optimal function  $f^*$ explicitly as the normalised eigenfunction  corresponding to the largest eigenvalue $\lambda$ where
\be\label{eq-odeforfstar}
(L_K^B+{F_K^B}'(1,x))f^*(x)=\lambda f^*(x) \ \ \ \text{in} \ (0,K) ,
\e
and we find that, in fact, $\lambda=\lambda(K)=(m-1)\beta-\mu^2/2-\pi^2/2K^2$ and
\be\label{eq-fstar}
f^*(x)  \propto \frac{\sin(\pi x/K)}{p_K(x)} e^{\mu x}, \ \ x\in(0,K),
\e 
up to a normalising constant.
Then we find the 'optimal' occupation density as 
\bex
\Pi^{B,*}_K(x):=(f^*(x))^2\Pi^B_K(x)= \frac{2}{K}\sin^2(\pi x/K), \ \ \ x\in(0,K).
\exÊ
In summary, we guess that the expected growth rate of the number of particles in the blue tree is  $\lambda(K)$ and that
\be\label{eq-largedevoptimalf}
\lambda(K) = \int_0^K (L^B_K+ {F_K^B}'(1,y) )f^*(y) \ f^*(y) \Pi^B_K(y) \ dy.  
\e 
We would  anticipate that the a.s. growth rate is also $\lambda(K)$ in agreement with the expected growth rate.\\
{\it Step (ii)(Upper bound on $\lambda(K)$)} The term $L^B_K f^* f^*$ is non-positive as it represents the cost of spending time like  $(f^*)^2 \Pi_K^B$, hence omitting it will give an upper bound for $\lambda(K)$, that is
\bex
\ \ \int_0^K {F_K^B}'(1,y) \Pi_K^B(y) \ dy &\leq& \lambda(K) 
\ex
{\it Step (iii)(Lower bound on $\lambda(K)$)} By taking $f=1$ in (\ref{eq-largedevoptimalf}), we  get a lower bound on $\lambda(K)$ since $f^*$ maximizes the expression in (\ref{eq-largedeviationexpecnumber}). Thus
\bex
\lambda(K) &\leq& \int_0^K {F_K^B}'(1,y) \Pi_K^{B,*}(y) \ dy.
\ex
{\it Step (iv) (Asymptotics)} By Theorem \ref{theo-survivalprobrough}, $ p_K(x)\sim  c_K \sin(\pi x/K_0) e^{\mu x} $, as $K\downarrow K_0$, and  we can easily deduce that  $\Pi^B_K(x) \sim \Pi_{K_0}^{B,*}(x)$,  as $K\downarrow K_0$.  We will make rigorous later that ${F_K^B}'(1,x) \sim (m-1) \beta c_K \sin(\pi x/K_0) e^{\mu x}$ as $K\downarrow K_0$. Our conjecture is therefore that
\bex
\lambda(K)  \sim c_{K} \frac{2  \beta}{K_0} \int_0^{K_0}   \sin^3(\pi y/K_0) e^{\mu y} \ dy, \text{ as } K\downarrow K_0.
\ex
Since we can calculate the integral explicitly this gives an exact asymptotic for $c_K$  which agrees with the one given in Proposition \ref{prop-constant} and Theorem \ref{theo-survivalprob}. 

\subsubsection{Proof of  Proposition \ref{prop-constant}}\label{sec-rigorous}
Let us now come to the rigorous proof. Recall that the backbone $(X,\mbf{P}^{B,K})$ is the process constructed in steps (i)-(iii) in Theorem \ref{theo-decomposition}, which was further characterised in Proposition \ref{defi-bluetree}.   First, we need to confirm the conjecture that the number of particles 
in $(X,\mbf{P}^{B,K}_x)$  grows at rate $\lambda(K)$. 
\begin{prop}\label{prop-growthofblue}
For $x\in(0,K)$,
\bex
\lim_{t\to\infty}\frac{1}{t}{ \log |N_t| } = {\lambda(K) }, \qquad \mbf{P}_x^{B,K} \text{-a.s.}
\ex
\end{prop}
Step (i) of the heuristic suggests that the growth rate of the backbone is $\lambda(K)$ and, moreover, that it can be expressed as (\ref{eq-largedevoptimalf}).
The idea of the rigorous proof of Proposition \ref{prop-growthofblue} is now to construct a martingale of the form in (\ref{eq-additivemartingalegeneral}) with $\hat{\Upsilon}$ built from $f^*$ in (\ref{eq-fstar}). We then find upper and lower bounds for this martingale which will, in turn, give bounds on the growth of the number of particles. Before we do this we prove an auxiliary result on the boundedness of $f^*$.
\begin{lem}\label{lem-fstarbounded}
The function $f^*$ is uniformly bounded in $(0,K)$. 
\end{lem}

\begin{proof}
The function $f^*$ is continuous in $(0,K)$ and it is therefore sufficient to show that $\limsup_{x\downarrow 0}f^*(x)$ and $\limsup_{x\uparrow K}f^*(x)$ are bounded.\\
An application of L'H\^opitals rule gives
\be\label{eq-sinbyotherstuff}
\lim_{x\downarrow 0} \frac{\sin(\pi x/K) e^{\mu x}}{ \frac{\pi}{-2 K \mu} (e^{-2\mu x}-1)}= 1
\e
To conclude that $\limsup_{x\downarrow 0}f^*(x)<\infty$, it therefore suffices to show that there exists a constant $c>0$ such that
\bex
c \   \frac{1}{-2\mu}
(1-e^{2\mu x}) \leq  \ p_K(x),  \ \ \ \text{ for all } x \text{ sufficiently close to zero.}
\ex
By Remark \ref{rem-fkpp},  $(\prod_{u\in N_t} (1-p_K(x_u(t))),t\geq 0)$ is a $P_x^K$-martingale and it follows then by a standard Feynman-Kac argument that $1-p_K(x)$ satisfies
\bex
1-p_K(x) = 1 + \mbb{E}_x^K \int_0^{\tau_{(0,K)}} F(1-p_K(\xi_s)) \ ds, \ \ \ x\in(0,K),
\ex
where  $\tau_{(0,K)}$ is the first time $\xi$ exists the interval $(0,K)$. To compute the expectation above we use the potential density of $\xi$, see e.g. Theorem 8.7 in \cite{kyprianoubook},  
and we get
\be\label{eq-integralspotential}
 - p_K(x)&=& \mbb{E}_x^K \int_0^{\tau_{(0,K)}} F(1-p_K(\xi_s)) \ ds \nonumber \\
 &=& \frac{1}{- \mu}(e^{-2\mu x}-1) \int_0^K F(1-p_K(y)) \frac{(e^{-2\mu (K-y)}-1)}{(e^{-2\mu K}-1)} dy \nonumber \\
  && \ \ \ \ \ \  \ + \frac{1}{\mu} \int_0^K F(1-p_K(y)) (e^{-2\mu(x-y)}-1) \ dy. 
\e  
Since $F(s) < 0$ for $0<s<1$, the first integral in the last equality on the right-hand side of (\ref{eq-integralspotential}) is strictly negative. Regarding boundedness  of this integral it is clear that the integrand is bounded for $y$ near $K$. By an application of L'H\^opitals rule it follows that the integrand is also bounded near $0$. Hence we can set
\bex
c:= - \int_0^K F(1-p_K(y)) \frac{(e^{-2\mu (K-y)}-1)}{(e^{-2\mu K}-1)} dy >0.
\ex 
With the second integral in the last equality on the right-hand side of (\ref{eq-integralspotential}) being non-positive, for $x$ close to $0$, we get
\bex
p_K(x) \geq  2c  \frac{1}{- 2\mu}(e^{-2\mu x}-1) \   \ \ \ \text{ for all } x \text{ sufficiently close to zero.}
\ex
which, by (\ref{eq-sinbyotherstuff}), gives the desired result.\\
To establish boundedness as $x$ approaches $K$, 
we observe that $ p_K(x) =  \bar{p}_K(K-x)$, where $\bar{p}_K$ denotes the survival probability for a branching diffusion which evolves as under $P^K_x$ but with positive drift $\mu$. 
Similarly to the previous argument we can then show that there exists a constant $c>0$ such that  $c \bar{p}_K(K-x)\geq \sin(\pi x/K)  e^{\mu x}$, for $x$ sufficiently close to $K$, which finishes the proof. 
\end{proof}

\begin{proof}[Proof of Proposition \ref{prop-growthofblue}]
We proof the upper bound by contradiction. Recall the embedding procedure described in Section \ref{sec-spine}. Choose $\epsilon$ small enough such that $\lambda(K+2\epsilon)>0$.  Choose a $\delta>0$ and suppose that there exists an increasing (random) sequence  $t_n$, $n=1,2,...$, which tends to infinity, such that  ${\log N_{t_n}}$ is bigger than ${ (\lambda(K+2\epsilon)+\delta) {t_n}}$, for any $n\in \mbb{N} $ under $P^K$ . Then, under $P^{(\epsilon,K+\epsilon)}$,
\bex
Z ^{(\epsilon,K+\epsilon)}(t) &:=& \sum_{u\in N_t} \sin(\pi (x_u(t)+\epsilon)/(K+2\epsilon)) e^{-\lambda(K+2\epsilon) t} e^{\mu (x_u(t)+\epsilon)} \\
&\geq& \left.|N_t\right|_{(0,K)}| \ e^{-\lambda(K+2\epsilon)t} \times  \sin(\pi \epsilon/(K+2\epsilon)).
\ex
Since we assumed that, along the sequence $t_n$, $n=1,2,...$, the number of particles $\left .|N_{t_n}\right|_{(0,K)}|$ is bounded from below by $\exp\{(\lambda(K+2\epsilon) +\delta) {t_n} \}$, the right-hand side above tends to infinity along this sequence as ${n}\to\infty$ which contradicts Proposition \ref{theo-l1convergence}. As we can take $\epsilon$ and $\delta$ arbitrary small we obtain $\limsup_{t\to\infty} (\lambda(K) t)^{-1} \log |N_t| \leq 1$, under $P^{K}$. 
By the thinning argument in Corollary \ref{rem-thinning}, we immediately get that $\lambda(K)$ is also an upper bound for the growth rate of $|N_t|$ under $\mbf{P}^{B,K}$. \\
For the lower bound, as alluded to above, we begin with constructing a $\mbf{P}^{B,K}$-martingale of the form (\ref{eq-additivemartingalegeneral}). Since $f^*$ satisfies $(L_K^B+{F_K^B}'(1)-\lambda(K))f^*=0$, it follows by an application of It\^o's formula that
\bex
f^*(\xi_t) e^{\int_0^t (F'(\xi_s,1) - \lambda(K)) ds}, \ \ t\geq 0
\ex
is a martingale with respect to $\sigma(\xi_t, t\geq 0)$, where $(\xi,\mbb{P}^{B,K})$ is 
an $L^B_K$-diffusion. Appealing to the discussion in  Remark \ref{rem-generalmartingale}
we then see that 
\bex
M_{f^*}(t) =\sum_{u\in N_t} f^*(x_u(t)) e^{-\lambda(K) t}, \ \ t\geq 0,
\ex
is  a $\mbf{P}^{B,K}_x$-martingale. \\
The proof of $L^1(\mbf{P}_x^{B,K})$-convergence of $M_{f^*}$ follows by a classical spine decomposition argument (cf. the proof of Theorem 1 in \cite{kyprianou}, as well as the proofs in \cite{lyons} and  \cite{lyonsetal}) and is therefore omitted. $L^1(\mbf{P}_x^{B,K})$-convergence implies then that $\mbf{P}_x^{B,K}(M_{f^*}(\infty)>0)>0$. 
Now set $g(x):=\mbf{P}_x^{B,K}(M_{f^*}(\infty) = 0)$, for $x\in(0,K)$. Then the product
\bex
\pi^{g}(t)= \prod_{u\in N_t} g(x_u(t)), \ \ \ t\geq 0,
\ex is a $\mbf{P}_x^{B,K}$-martingale with almost sure limit $\mbf{1}_{\{M_{f^*}(\infty)=0\}}$ (cf. proof of Proposition \ref{theo-l1convergence}). Therefore we have
\bex
g(x)=\mbf{E}_x^{B,K}(\pi^{g}(t)) 
\leq  {\mbb{{E}}}^{B,K}_x(g(\xi_t)), \ \ \ \text{for all} \ x\in(0,K).
\ex
Hence we conclude that the process $(g(\xi_t), t\geq 0)$ is a $[0,1]$-valued ${\mbb{P}}^{B,K}_x$-submartingale and  it converges ${\mbb{P}}^{B,K}_x$-a.s. to a limit $g_\infty$. However $\xi$ is  positive recurrent  under $\mbb{P}^{B,K}_x$ and thus $g(\xi_t)$ can only converge if it is constant, hence $g(x)=g_\infty$ for  all $x\in(0,K)$. Since $0\leq g\leq 1$ we then have $g_\infty \in [0,1]$. 
Assume now that $g_\infty\in[0,1)$. Note that, under $\mbf{P}^{B,K}_x$, $|N_t|$ tends to infinity as $t\to\infty$ since each particle in $(X,\mbf{P}^{B,K}_x)$ is replaced by at least two offspring when it dies and there is no killing. Thus we get $\pi^g(t)\to 0$, ${\mbf{P}}^{B,K}_x$-a.s. and therefore $g(x)= {\mbf{E}}^{B,K}_{x}(\pi^g(\infty))=0$. 
In conclusion, $g$ is  identical to either $0$ or $1$.  But we already know that the martingale limit $M_{f^*}(\infty)$ is strictly positive with positive probability and consequently, $g(x)=\mbf{P}^{B,K}_x(M_{f^*}(\infty)=0)=0$, for all $x\in(0,K)$. We may now conclude that $\liminf_{t\to\infty}\log {M_{f^*}(t)}/\lambda(K) t \geq 0$ $\mbf{P}_x^{B,K}$-a.s.  \\ 
We  look for an upper bound on $M_{f^*}(t)$ which will, in turn, provide    a lower bound on $|N_t|$ under $\mbf{P}^{B,K}_x$. By Lemma \ref{lem-fstarbounded}, $f^*$ is bounded by a constant $c>0$ in $(0,K)$. Thus, under $\mbf{P}^{B,K}_x$, for $t\geq 0$, 
\bex
M_{f^*}(t)\leq c \ |N_t| \ e^{-\lambda(K) t},
\ex
and we see that, $ \mbf{P}^{B,K}_x\text{-a.s.}$,
\bex
\liminf_{t\to\infty}\frac{\log |N_t|}{\lambda(K) t} \geq \liminf_{t\to\infty} \frac{\log M_{f^*}(t) - \log c + \lambda(K) t}{\lambda(K) t}\geq 1, 
\ex 
which completes the proof.
\end{proof}
In step (ii) of the heuristic we claimed  that $(L_K^B f^*)f^*$ is non-positive to get an upper bound on  $\lambda(K)$ which is essentially what we will now do. 
 Recall that the invariant density for the infinitesimal generator $L_K^B$ respectively the optimal occupation density of  $(X,\mbf{P}^{B,K}_x)$ are given by
\be\label{defi-invariantmeasures}
\Pi^B_K(y) &=& \frac{p_K(y)^2 e^{-2\mu y}}{\int_0^K p_K(z)^2 e^{-2\mu z} \ dz} \nonumber \\
 \text{ and } \ \ \Pi_K^{B,*}(y) &:=& (f^*(y))^2\ \Pi_K^B(y) = \frac{2}{K}\sin^2(\pi y/K). 
\e 
\begin{lem}\label{lem-growthupperbound}
For $K>K_0$, we have
\bex
\lambda(K)\leq \int_0^K {F_K^B}'(1,y)  \ \Pi_K^{B,*}(y) \ dy.
\ex
\end{lem}

\begin{proof}
We have   $\int_0^K (f^*(y))^2 \Pi_K^B(y) \ dy = \int_0^K 2/K \sin^2(\pi x/K)\ dx =1$. Then multiplying by $\lambda(K)$ gives
\bex
\lambda(K) &=& \int_0^K \lambda(K) f^*(y) f^*(y) \Pi_K^B(y) \ dy 
\ex
Recall that $f^*$ is given by (\ref{eq-fstar}) and satisfies the ODE in (\ref{eq-odeforfstar}). Thus we can replace the term $\lambda(K) f^*$ above by $(L_K^B+{F_K^B}'(1,y))f^*(y)$. Therefore
\be\label{eq-estimateforlambdac}
\lambda(K) &=&  \int_0^K \left(\frac{1}{2}(f^*(y))'' - (\mu - \frac{p_K(y)'}{p_K(y)}) (f^*(y))'\right) f^*(y) \Pi_K^B(y) \ dy \nonumber \\ && \qquad \qquad + \int_0^K {F_K^B}'(1,y) (f^*(y))^2 \Pi_K^B(y)\ dy. 
\e
Noting that $\Pi_K^{B,*}(y)= (f^*(y))^2 \Pi_K^B(y)$, the result then follows if we can show that the first integral in (\ref{eq-estimateforlambdac}) is non-positive.\\
We use integration by parts for the first term in the first integral in (\ref{eq-estimateforlambdac}) to get 
\be\label{eq-intbyparts}
\int_0^K \frac{1}{2}(f^*(y))'' f^*(y) \Pi_K^B(y) \ dy \hspace{70 mm} &&\nonumber \\
= \frac{1}{2} \Big( [(f^*(y))' f^*(y)\Pi_K^B(y)]_0^K \hspace{60 mm} && \nonumber \\
      \  - \int_0^K (f^*(y))' \left((f^*(y))' \Pi_K^B(y) \ dy  + f^*(y) (\Pi_K^B(y))' \ dy\right)\Big). && \nonumber \\
\e
We want to show that the first term on the right-hand side above is zero. By Lemma \ref{lem-fstarbounded}, $f^*$ takes a finite value at $0$ and $K$ and hence it suffices to show that $(f^*)'\Pi_K^B$ evaluated at $0$ and $K$ is zero. By simply differentiating $f^*$ and recalling that $\Pi_K^B(y)\propto p_K(y)^2 e^{-2\mu y}$ we get
\bex
(f^*(y))' \Pi_K^B(y) &  \propto& e^{- \mu y} \Big( (\mu \sin(\pi y/K) + \frac{\pi}{K} \cos(\pi y/K))p_K(y)  \\
&&  \qquad\qquad\qquad\qquad - \sin(\pi y / K) p_K'(y) \Big).
\ex
Differentiating both sides of equation (\ref{eq-integralspotential}) with respect to $x$, it is easily seen that  $p_K'(x)$ is bounded for all $x\in[0,K]$. Therefore $(f^*(y))' \Pi_K^B(y)$ is equal to $0$ at $0$ and $K$ and thus the first term on the right-hand side of (\ref{eq-intbyparts}) vanishes.\\
The first integral in (\ref{eq-estimateforlambdac}) now becomes
\be\label{eq-lotsofintegrals}
&& \int_0^K \left(\frac{1}{2}(f^*(y))'' - (\mu - \frac{p_K'(y)}{p_K(y)}) (f^*(y))'\right) f^*(y) \Pi_K^B(y) \ dy \qquad\qquad\qquad \nonumber \\
&&\qquad = - \frac{1}{2} \int_0^K \left(((f^*(y))')^2 -  (f^*(y))' f^*(y)  (\Pi_K^B(y))' \right)\Pi_K^B(y) \ dy \nonumber \\
&& \qquad\qquad  \qquad - \int_0^K (\mu-\frac{p_K'(y)}{p_K(y)}) f^*(y) (f^*(y))' \Pi_K^B(y) \ dy.  
\e
Differentiating $\Pi_K^B$ gives $(\Pi_K^B)'=-2(\mu-\frac{p_K'}{p_K}) \Pi_K^B$. Thus, in the right-hand side of (\ref{eq-lotsofintegrals}), the second term in the first integral cancels with the second integral and we arrive at
\bex
&& \int_0^K \left(\frac{1}{2}(f^*(y))'' - (\mu - \frac{p_K'(y)}{p_K(y)}) (f^*(y))'\right) f^*(y) \Pi_K^B(y) \ dy \\
&=& -\frac{1}{2} \int ((f^*(y))')^2 \Pi_K^B(y) \ dy, 
\ex
which is less than or equal to zero and the proof is complete.
\end{proof}

Step (iii) of the heuristic claims that we can lower bound $\lambda(K)$ by replacing $f^*$ in ({\ref{eq-fstar}}) with $f=1$.
This suggests to modify the martingale argument in the proof of Proposition  \ref{prop-growthofblue}  using a martingale of the form (\ref{eq-additivemartingalegeneral}) with $\hat{\Upsilon}$ built from the constant function $\mbf{1}$. 
\begin{lem}\label{lem-growthlowerbound}
For $x\in(0,K)$,
\bex
\lambda(K) \geq \int_0^K  {F_K^B}'(1,y)  \Pi_K^B(y) \ dy.
\ex
\end{lem}
\begin{proof}
The constant process $\hat{\Upsilon}=(\hat{\Upsilon}(t)=1, t\geq 0)$ is a trivial martingale with respect to $\sigma(\xi_t, t\geq 0)$, where $(\xi,\mbf{P}^{B,K})$ is a  $L_K^B$-diffusion. Thus according to Remark \ref{rem-generalmartingale}, the process  $M_1=(M_1(t),t\geq 0)$  
defined by 
\bex
M_1(t) =\sum_{u\in N_t^B} e^{-\int_0^t {F_K^B}'(1,x_u(s)) \ ds}, \ \ \ t\geq 0,
\ex 
is a $\mbf{P}_x^{B,K}$-martingale. $L^1(\mbf{P}^{B,K}_x)$-convergence and uniform integrability of $M_1$ again
follow by a spine decompisition argument in the manner of the proof of Theorem 1 in \cite{kyprianou}. $L^1(\mbf{P}_x^{B,K})$-convergence implies that $\mbf{P}^{B,K}_x(M_1(\infty)>0)>0$ and repeating the argument in the proof of Proposition \ref{prop-growthofblue} we immediately see that  $\mbf{P}_x^{B,K}(M_1(\infty)>0)=1$. Therefore, we conclude that $\liminf_{t\to\infty} \log M_1(t)/\lambda(K) t\geq 0$, $\mbf{P}_x^{B,K}$-a.s.\\ 
On the filtration  $\mcal{F}_t=\sigma(X_s, s\leq t)$ we define a change of measure by
\bex
\left.\frac{dP_x^{M_1,K}}{d\mbf{P}_x^{B,K}}\right|_{\mcal{F}_t} = \frac{M_1(t)}{M_1(0)}, \ \ t\geq 0. 
\ex
Since the martingale $M_1$ is of the form in (\ref{eq-additivemartingalegeneral}), it induces as spine decomposition  which, according to Remark \ref{rem-generalmartingale}, reads as follows. Under $P_x^{M_1,K}$, the spine $\xi$ is an $L^B_K$-diffusion and along its path we immigrate independent copies of the $\mbf{P}^{B,K}$-branching diffusion (we do not need to specify the rate of immigration and the distribution of the number of immigrants since they will not be relevant).\\ 
Next, fix an $\epsilon>0$ and define, for $t\geq 0$ and each $u\in N_t^B$, the set
\bex
A_t^{u} = \left\{\left|\frac{1}{t} \int_0^t {F_K^B}'(1,x_u(s)) 
\ ds - \int_0^K {F_K^B}'(1,y) 
\ \Pi_K^B(y) \ dy \right|<\epsilon\right\}
\ex
and consider the process we obtain from $M_1$ by considering the particles in $A_t^{u}$ only, that is 
\bex
\tilde{M}_1(t) = \sum_{u\in N_t^B} \mbf{1}_{A_t^{u,\epsilon}} \  e^{-\int_0^t {F_K^B}'(1,x_u(s)) \ ds }, \ \ \ t\geq 0.
\ex
Let  $A_t^{\xi}$ be the event we get if we simply replace $x_u(t)$ by the spine process $\xi_t$ in the definition of $A_t^{u}$. Since $(\xi,P^{M_1,K})$ has invariant density $\Pi_K^B$ we have $\mbf{1}_{A_t^{\xi}} \to \mbf{1}$ ${P}^{M_1,K}_x$-a.s.   
By Theorem 1 in \cite{harrisrobertsnoteonspinemartingales}, $\tilde{M}_1$ therefore has the same limit as $M_1$ under ${P}^{M_1,K}_x$, and moreover, since $M_1$ is uniformly integrable, this also holds true under $\mbf{P}_x^{B,K}$. In particular we have  
\bex \liminf_{t\to\infty} \frac{\log \tilde{M}_1(t)}{\lambda(K) t} = \liminf_{t\to\infty} \frac{\log {M}_1(t)}{\lambda(K) t} \geq 0,  \ \  \mbf{P}_x^{B,K}\text{-a.s.} 
\ex
For $t\geq 0$, we now get an upper bound for $\tilde{M}_1(t)$ under $\mbf{P}_x^{B,K}$ by 
\bex
\tilde{M}_1(t)
&\leq & |N_t| e^{-t \int_0^K {F_K^B}'(1,y) 
\ \Pi_K^B(y) \ dy - \epsilon}. 
\ex
Consequently,  $\mbf{P}_x^{B,K}\text{-a.s.}$,
\bex
&& \liminf_{t\to\infty}\frac{\log |N_t|}{t\left(\int_0^K {F_K^B}'(1,y)
\Pi_K^B(y) \ dy \right)}  \\
&\geq& \liminf_{t\to\infty}\frac{\log \tilde{M}_1(t) + t\left(\int_0^K {F_K^B}'(1,y) \Pi_K^B(y)\ dy - \epsilon\right)}{t\left(\int_0^K {F_K^B}'(1,y) \Pi_K^B(y) \ dy \right)} \\
&\geq& \frac{\int_0^K {F_K^B}'(1,y) \Pi_K^B(y) \ dy- \epsilon}{\int_0^K {F_K^B}'(1,y) \Pi_K^B(y) \ dy}, 
\ex
and taking $\epsilon\downarrow 0$ gives the result.
\end{proof}

\begin{proof}[Proof of Proposition \ref{prop-constant}]
By  Lemma  \ref{lem-growthupperbound} and \ref{lem-growthlowerbound},  we get the following bounds on $\lambda(K)$
\be\label{eq-upperlowerbound}
\int_0^K {F_K^B}'(1,y) \Pi_K^B(y) \ dy  \leq \lambda(K) \leq \int_0^K {F_K^B}'(1,y)
\Pi_K^{B,*}(y) \ dy, \nonumber \\
\e
where  $\Pi_K^B$  and $\Pi_K^{B,*}$ were defined in (\ref{defi-invariantmeasures}). 
By Proposition \ref{prop-survivalprobasymptotics}, we have, as $K\downarrow K_0$,
\be\label{eq-pibispibstar}
\Pi_K^B(y)=\frac{p_K(y)^2 e^{-2\mu y}}{\int_0^K p_K(z)^2 e^{-2\mu z} dz} \sim \frac{2}{K_0} \sin^2 (\pi y/K) = \Pi_{K_0}^{B,*}(y), \nonumber \\
\e
where we have used that the asymptotics in  Proposition \ref{prop-survivalprobasymptotics} hold uniformly to deal with the integral in the denominator of the second term in (\ref{eq-pibispibstar}). The uniformity in  Proposition \ref{prop-survivalprobasymptotics} also ensures that (\ref{eq-pibispibstar}) holds uniformly for all $y\in(0,K_0)$.  
Further
\bex \lim_{s\uparrow 1}\frac{F(s)}{s(s-1)} = \lim_{s\uparrow 1}\frac{\beta (\sum_{n \geq 2} q_n s^n-1)}{s-1} = \lim_{s\uparrow 1} \beta \sum_{n\geq 2} q_n n s^{n-1} = (m-1) \beta,
\ex 
where we applied  L'H\^opitals rule in the second equality above.
 Then, together with  Proposition \ref{prop-survivalprobasymptotics},  as  $K\downarrow K_0$,
\be\label{eq-limfb}
{F_K^B}'(1,y) &=& (m-1)\beta + \frac{F(1-p_K(y))}{p_K(y)} \nonumber\\
&\sim& (m-1) \beta p_K(y) 
\sim (m-1) \beta c_K \sin(\pi y/K_0) e^{\mu y}. 
\e
Note that 
$
\frac{F(1-p_K(y)) }{p_K(y)} = - \frac{F(1)-F(1-p_K(y))}{1-(1-p_K(y))}. 
$
Convexity of $F$ yields then that $\left|\frac{F(1-p_K(y)) }{p_K(y)}\right|$ is bounded by $(m-1)\beta$. Thus $|{F_K^B}'(1,y)|\leq 2 (m-1)\beta$ and we can appeal to bounded convergence 
as we take the limit in (\ref{eq-upperlowerbound}). With (\ref{eq-pibispibstar}) and (\ref{eq-limfb}) we get
\bex
\lambda(K)
&\sim& c_K \frac{2 (m-1) \beta}{K_0} \int_0^{K_0} \sin^3(\pi y/K_0) e^{\mu y} \ dy,\ \ \text{ as }  K\downarrow K_0.
\ex 
Evaluating the integral gives
\bex
\lambda(K)
\sim c_K  \frac{ 12 \ (m-1)\beta \ \pi^3 \ (e^{\mu K_0}+1)}{(K_0^2\mu^2+\pi^2)(K_0^2 \mu^2 + 9 \pi^2)}, \ \ \text{ as } K\downarrow K_0.
\ex
Finally, $\lambda(K)\sim \pi^2 (K-K_0) K_0^{-3}$ as $K\downarrow K_0$ which follows from the linearisation 
\bex
\lambda(K) &=& (m-1) \beta - \mu^2/2 - \pi^2/2K^2 \\
&=& \frac{\pi^2}{2K_0^2}- \frac{\pi^2}{2K^2} \\
&=& \frac{ \pi^2  (K-K_0+K_0)^2}{  2 K_0^2K^2} - \frac{ \pi^2 K_0^2}{ 2 K_0^2K^2} \\
&=& \frac{ \pi^2  (K-K_0)}{  K_0 K^2} - \frac{ \pi^2 (K-K_0)^2}{ 2 K_0^2K^2}
\ex
and noting that the first term in the last line is the leading order term as $K\downarrow K_0$. This completes the proof.
\end{proof}

\section{Proof of Theorem \ref{theo-quasistationary}}\label{sec-backbonecriticality}

\begin{proof}[Proof of Theorem \ref{theo-quasistationary}]
Recall that  $(X,\mbf{P}^{D,K})$ was defined as the process $(X,P^K)$ conditioned on the event of survival  and characterised via the change of measure in (\ref{eq-changemeasurepd}) and Theorem \ref{theo-dressedbackbone}. 
\\
Fix a $K'>K_0$ and further denote by $\left.N_t\right|_{(0,K)}$ the set of particles whose ancestors (including themselves) have not exited $(0,K)$ up to time $t$. Then, for $K\leq K'$, and for $x\in(0,K_0)$ and $A\in\mcal{F}_t$, 
\bex
\lim_{K\downarrow K_0} \mbf{P}^{D,K}_x (A) 
&=&  E^{K'}_x\left( \mbf{1}_A \ \lim_{K\downarrow K_0} \frac{1-\prod_{u\in \left.N_t\right|_{(0,K)}} (1-p_K(x_u(t)))}{p_K(x)}\right) \nonumber,
\ex
since $\left.N_t\right|_{(0,K)}$ has the same law under $P^{K}$ and $P^{K'}$. Suppose the particles in $\left. N_t\right|_{(0,K)}$ are ordered, for instance according to their spatial positions,  and we write $u_1,...,u_{\left.N_t\right|_{(0,K)}}$. We can now expand the term within the expectation on the right-hand side as
\be\label{eq-expansion}
&& \frac{1-\prod_{u\in \left. N_t\right|_{(0,K)}} (1-p_K(x_u(t)))}{p_K(x)} \nonumber \\ 
 && \qquad = \sum_{i = 1}^{|\left. N_t\right|_{(0,K)}|}  \frac{p_K(x_{u_i}(t))}{p_K(x)} \ \prod_{j<i} (1-p(x_{u_j}(t))) 
\e
which is bounded from above by $|\left. N_t\right|_{(0,K)}| (p_K(x))^{-1}$. Recall the asymptotics for $p_K$ in Theorem \ref{theo-survivalprobrough} and in particular Lemma \ref{lem-survivalsin}, noting that these results hold uniformly in $(0,K_0)$. Since $|\left. N_t\right|_{(0,K)}|$ has finite expectation, we can apply the Dominated convergence theorem to the expression in (\ref{eq-expansion}), and  we get 
\bex
&& E^{K'}\Big( \mbf{1}_A \lim_{K\downarrow K_0}   \sum_{i = 1}^{|\left.N_t\right|_{(0,K)}|}  \frac{p_K(x_{u_i}(t))}{p_K(x)} \ \prod_{j<i} (1-p(x_{u_j}(t))) \Big) \\
&&\qquad = E^{K_0} \Big( \mbf{1}_A \ \sum_{i=1}^{|\left. N_t\right|_{(0,K_0)}|} \frac{\sin(\pi x_{u_i}(t))/K_0)e^{\mu x_{u_i}(t)}}{\sin(\pi x /K_0) e^{\mu x}}  \Big).
\ex
Hence, for $A\in{\mcal{F}_t}$, we arrive at
\bex
\lim_{K\downarrow K_0} \mbf{E}^{D,K}_x(A)
&=& E^{K_0}_x\left(\mbf{1}_A \ \frac{\sum_{u\in N_t} \sin(\pi x_u(t)/ K_0) e^{\mu x_u(t)} }{\sin(\pi x/ K_0) e^{\mu x}} \right) \\
&=& E^{K_0}_x\left(\mbf{1}_A \ \frac{Z^{K_0}(t) }{Z^{K_0}(0)}\right), 
\ex
where $Z^{K_0}$ is the martingale used in the change of measure in (\ref{changemeasure-z}) in Section \ref{sec-spine}. The evolution under this change of measure is described in the paragraph following (\ref{changemeasure-z}) and agrees with that of $(X^*,P^*_x)$ as defined in Theorem \ref{theo-quasistationary}.

\end{proof}

\section{Super-Brownian motion in a strip}\label{sec-superbrownianmotion}
Recall from (\ref{eq-generatorl}) that 
the infinitesimal generator $L$ is given as
$L=\frac{1}{2} \frac{d^2}{dx^2} -\mu \frac{d}{dx}, \ \ x\in(0,K),  $
defined for all functions $u\in C^2(0,K)$ with u$(0+)=u(K-)=0$. 
Changing the domain to $u\in C^2(0,K)$ with $u''(0+)=u''(K-)=0$, then $L$ corresponds to  Brownian motion with absorption (instead of killing) at $0$ and $K$. For technical reason,  we will assume from now on that $\mcal{P}^K=\{\mcal{P}_t^K,t\geq 0\}$ is the corresponding diffusion semi-group of Brownian motion with absorption and is therefore conservative. Note that all the results presented for branching Brownian motion with killing at $0$ and $K$ also hold in the setting of absorption at $0$ and $K$ when we restrict the process with absorption to particles within $(0,K)$, in particular when defining $N_t$ as the number of particles who are alive and have not been absorbed at time $t$. \\
Suppose $Y=(Y_t,t\geq 0)$ is a Super-Brownian motion with associated semi-group $\mcal{P}^K$ and branching mechanism $\psi$ of the form 
\bex
\psi(\lambda) = - \alpha \lambda + \beta \lambda^2 + \int_0^\infty (e^{-\lambda y} - 1 + \lambda y ) \  \Pi(dy),\ \  \lambda\geq 0, 
\ex
where $\alpha=-\psi'(0+)\in (0,\infty)$, $\beta\geq 0$ and $\Pi$ is a measure concentrated on $(0,\infty)$ satisfying $\int_{(0,\infty)} (x\wedge x^2) \ \Pi(dx)<\infty$. For an initial configuration $\eta\in\mcal{M}_f(0,K)$, the space of finite measures supported on $(0,K)$, we denote the law of $Y$ by $\tilde{P}_\eta^K$. 
The existence of this class of superprocesses follows from \cite{dynkin}. \\
Since $\alpha=-\psi'(0+)>0$, the function $\psi$ is the branching mechanism of a supercitcal continuous-state branching process (CSBP), say $Z$. We assume henceforth that $\psi$ satisfies the non-explosion condition $\int_{0+} |\psi(s)|^{-1}\ ds = \infty$ and further that $\psi(\infty)=\infty$. The last condition,  together with $\psi'(0+)<0$, ensures that 
$\psi$ has a unique positive root $\lambda^*$. The parameter $\lambda^*$ is the survival rate of $Z$ in the sense that the probability of the event $\{\lim_{t\to \infty} Z_t = 0\}$ given $Z_0=x$ is  $e^{- \lambda^*x}$, which is  strictly positive. We further assume from now on that  $
\int^{+\infty} (\psi(s))^{-1} \ ds < \infty,
$
which guarantees that the event $\{\lim_{t\to\infty} Z_t=0 \}$ agrees with the event of extinction, that is $\{\exists  t >0 : Z_t=0\}$ a.s. This implies in turn that, for the Super-Brownian motion $Y$, the event of becoming extinguished and the event of extinction agree  $\tilde{P}^K$-a.s. We denote the event of extinction of $Y$  by $\mcal{E}=\{\exists t >0 : Y_t(0,K)=0\}$, where $Y_t(0,K)$ is the total mass within $(0,K)$ at time $t$. We can characterise the
$\tilde{P}_\eta^K$ -superdiffusion via its Laplace functional (see e.g. Section 4.1.1 in \cite{dynkin2}).
\begin{lem}\label{lem-deflaplacesbm}
For all $f\in B_+(0,K)$, 
\bex
\tilde{E}_\eta^K(e^{-\langle f, {Y}_t \rangle}) =  e^{- \la \tilde{u}_f^K(\cdot,t), \eta \ra } , \ \ \  \eta\in\mcal{M}_f(0,K), \ t\geq 0,
\ex 
where $\tilde{u}_f^K(x,t)$ is the unique non-negative solution to the semi-group equation
\be\label{eq-ufsbm}
\tilde{u}_f^K(x,t)= \mcal{P}_t^K[f(\cdot)] (x) - \int_0^t  \mcal{P}_{t-s}^K[\psi (\tilde{u}_f^K(\cdot,s))](x)\ ds.
\e
We call the function $\tilde{u}_f^K(x,t)$ the Laplace functional of $(Y,\tilde{P}_\eta^K)$.
\end{lem}
We have used the  notation $\la f, \eta \ra = \int_0^K f(x) \eta(dx)$, for $\eta \in\mcal{M}_f[0,K]$. We define the survival rate $w_K$ of the $\tilde{P}^K$-superdiffusion as the function satisfying
\be\label{eq-defsurvivalrate}
\tilde{P}^K_x(\mcal{E}) = \exp\{- w_K(x) \}, \ \ \text{ for } \eta\in\mcal{M}_f[0,K], \nonumber \\
\e
and, taking $f\equiv \theta$ constant in Lemma \ref{lem-deflaplacesbm} ,  we can deduce that 
\bex
-\log\tilde{P}^K_\eta( \mcal{E})  =  \lim_{t \to\infty} \lim_{\theta \to\infty} \la \tilde{u}_{\theta} (\cdot), \eta\ra  = \la w_K, \eta \ra.
\ex 
It can be derived, again by Lemma \ref{lem-deflaplacesbm}, that $w_K$ is a solution to
\be\label{eq-odeforw} 
L u - \psi(u) = 0 \ \ \text{ with } \ \ u(0)=u(K)=0.
\e 
Analogous to Theorem \ref{theo-criteriaforpossurvivalprob}, it is possible to give a necessary and sufficient condition for a positive survival rate which follows from a spine change of measure argument  in the spirit of Section \ref{sec-spine}, now using the $\tilde{P}^K_x$-martingale 
\be\label{eq-spinechangeofmeasuresbm}
\tilde{Z}^K(t)={\int_0^{K} \sin(\pi x /K) e^{\mu x - \lambda(K)t } Y_t(dx)  }
, \ \ \ t \geq 0,
\e
where here $\lambda(K)= -\psi'(0+)-\mu^2/2-\pi^2/2K^2$. Assuming henceforth in addition that  $\int_1^\infty x \log x \Pi(dx) < \infty$, one can then show, in the fashion of Kyprianou et al. \cite{klmr},  that $\tilde{Z}^K$ is an $L^1(\tilde{P}^K_x)$ martingale if and only if $\lambda(K)>0$. It can thus be concluded that $w_K$ is positive if $\lambda(K)>0$. \\
Let us now establish the connection between the $\tilde{P}^K$-superdiffusion and a $P^K$-branching diffusion via the following relations. 
Set 
\be\label{eq-relationbranchingmechs}
F(s)&=&\frac{1}{\lambda^*}\psi(\lambda^* (1-s)), \ \ s\in(0,1), \\
\label{eq-defprobextiguished} \bar{w}_K(x) &=& \lambda^* p_K(x),  \ \ \ \ \ \ \  \ \ \ \   x\in(0,K), \label{eq-relationsurvivalprobs}
\e
where $p_K$ is the survival probability of the $P^K$-branching diffusion. Bertoin et al. \cite{bfm}  show that   (\ref{eq-relationbranchingmechs}) is  the branching mechanism of a Galton-Watson process and they identify the Galton-Watson process with branching mechanism $F$ of (\ref{eq-relationbranchingmechs}) as  the backbone of the CSBP with branching mechanism $\psi$. If we can show that  $\bar{w}_K$ in (\ref{eq-defprobextiguished}) is indeed the survival rate $w_K$ then the following  Theorem is a direct consequence of Theorem \ref{theo-criteriaforpossurvivalprob} and 
Theorem \ref{theo-survivalprob}.
\begin{theo}\label{theo-survivalforsbm}
(i) If $\mu<\sqrt{-2\psi'(0+)}$ and $K>K_0$ where $K_0:= \pi (\sqrt{-2\psi'(0+)})^{-1}$, then $w_K(x)>0$ for all $x\in (0,K)$; otherwise $w_K(x)=0$ for all $x\in[0,K]$. \\
(ii) Uniformly for $x\in (0,K_0)$,  as $K\downarrow K_0$,
\be\label{eq-relationextinguishprob}
w_K(x) \sim \lambda^* (K-K_0) \frac{(K_0^2\mu^2+\pi^2)(K_0^2 \mu^2 + 9 \pi^2)}{12 \psi'(0+) \pi K_0^3(e^{\mu K_0}+1)}  \sin(\pi x/K_0) e^{\mu x}. \nonumber \\
\e
\end{theo}

\begin{proof}[Proof of Theorem \ref{theo-survivalforsbm}]
The relation in (\ref{eq-relationbranchingmechs})  gives $(m-1)\beta =F'(1-)=-\psi'(0+)$ and hence the $K_0$ in Theorem \ref{theo-survivalforsbm} is the same as the one in Theorem \ref{theo-criteriaforpossurvivalprob} and the $\lambda(K)$ defined earlier in this section agrees with $\lambda(K)$ as in Proposition \ref{theo-l1convergence}. In particular, $\lambda(K)>0$ if and only if  $\mu<\sqrt{-2\psi'(0+)}$ and $K>K_0$.\\
Suppose $\mu<\sqrt{-2\psi'(0+)}$ and $K>K_0$. By Remark  \ref{rem-uniqueness}, $p_K$ is the unique non-trivial solution to  $L u - F(1-u)=0$ on  $(0,K)$ with  $u(0)=u(K)=0$.  Using (\ref{eq-relationbranchingmechs}) it follows then that $\bar{w}_K$ given by (\ref{eq-defprobextiguished}) solves (\ref{eq-odeforw})  and the uniqueness carries over. That is, for $\mu<\sqrt{-2\psi'(0+)}$ and $K>K_0$, $\bar{w}_K$  is the unique non-trivial solution to (\ref{eq-odeforw}). On the other hand, we know that $w_K$ solves  (\ref{eq-odeforw}) and, by the spine argument we mentioned after (\ref{eq-spinechangeofmeasuresbm}), we know that $w_K$ is positive within $(0,K)$. By uniqueness, we have $\bar{w}_K = w_K$. \\
Suppose $\mu\geq \sqrt{-2\psi'(0+)}$ or $K\leq K_0$.  
Then  $p_K$ is identically zero and (\ref{fkpptw}) does not have a non-trivial solution. By the transformation in (\ref{eq-relationbranchingmechs}), the same holds true for (\ref{eq-odeforw}) and since $w_K$ is always a solution to (\ref{eq-odeforw}) it must be equal to zero. Thus (\ref{eq-defprobextiguished}) holds true again.\\
The result is now a direct consequence of Theorems \ref{theo-criteriaforpossurvivalprob} and \ref{theo-survivalprobrough}.
\end{proof}
We will now outline the backbone decomposition for the 
$\tilde{P}_\eta^K$-super -diffusion which consists of a copy of $(Y,\tilde{P}^K_\eta)$ conditioned on becoming extinct and further independent copies of $(Y,\tilde{P}^K)$ conditioned on becoming extinct which immigrate along a $\mbf{P}^{B,K}$-branching diffusion. Recall that the $\mbf{P}^{B,K}$-branching diffusion is also the backbone of the $P^K$-branching diffusion (Theorem \ref{theo-decomposition}).\\ 
Let us begin by studying the process $(Y,\tilde{P}^K)$ conditioned on becoming extinct. 
\begin{prop}\label{prop-redsuperdiffusion}
Define for $\eta\in\mcal{M}_f[0,K]$  and $t\geq 0$, 
\bex
\left. \frac{d \tilde{\mbf{P}}^{R,K}_\eta}{d \tilde{{P}}^{K}_\eta}\right|_{\tilde{\mcal{F}}_t} = \frac{e^{- \la w_K, Y_t \ra}}{e^{- \la w_K, \eta \ra}},
\ex
where $(\tilde{\mcal{F}}_t, t\geq 0 )$ is the natural filtration generated by $(Y,\tilde{P}^K_\eta)$. 
Then  $(Y,\tilde{\mbf{P}}^{R,K}_\eta)$ is equal in law to $(Y,\tilde{P}^K_\eta(\cdot|\mcal{E}))$. 
Further $(Y,\tilde{\mbf{P}}_\eta^{R,K})$ has spatially dependent branching mechanism 
\bex
\psi^{R,K}(s,x) =  \psi(s+w_K(x))-\psi(w_K(x)), \ \ \ s\geq 0 \text{ and } \ x\in[0,K] 
\ex
and diffusion semigroup $\mcal{P}^K$. 
\end{prop}
The proof of Proposition \ref{prop-redsuperdiffusion} is just a straightforward adaptation of the proof of Lemma 2 in \cite{bkm} and thus omitted. We point out that the motion of the $\tilde{\mbf{P}}^{R,K}$- superdiffusion remains unchanged and it is therefore different from the motion of the $\mbf{P}^{R,K}$-branching diffusion in Definition \ref{defi-bluetree}.
However, set $\tilde{u}_f^R(x,t)=\lambda^*(1-p_K(x))(1-u^R_f(x,t))$, where $u^R_f$ is the Laplace functional of the $\tilde{\mbf{P}}^{R,K}$-branching diffusion (Laplace functionals for branching diffusions are defined in a similar fashion to Lemma \ref{lem-deflaplacesbm}, see for instance Section 4.1.4 \cite{dynkin2}). Then together with the relations (\ref{eq-relationbranchingmechs}) and (\ref{eq-relationsurvivalprobs}) we can find that $\tilde{u}_f^R$ is the Laplace functional of the $\tilde{\mbf{P}}^{R,K}$-superdiffusion. Thus the $\tilde{\mbf{P}}^{R,K}$-superdiffusion can in this way be seen as the analogue of the $\mbf{P}^{R,K}$-branching diffusion in the superdiffusion setting. \\
We need to introduce some more notation before we can establish the backbone decomposition. Associated to the laws $\{\tilde{\mbf{P}}^{R,K}_{\delta_x},x\in[0,K]\}$ is the family of the so-called excursion measures $\{\mbb{N}_x^{R,K}, x\in [0,K]\}$, defined on the same measurable space, which satisfy 
\be\label{defi-nmeasure}
\mbb{N}_x^{R,K}(1-\exp\{-\la f , \ Y_t \ra \}) = - \log \tilde{\mbf{E}}^{R, K}_{\delta_x}(\exp\{-\la f , Y_t \ra \}),
\e
for any  $f \in B_+[0,K]$ and $t\geq 0$. These measures are formally defined and studied in Dynkin and Kuznetov \cite{dynkinkuznetsov}. Intuitively speaking, the branching property implies that $\tilde{\mbf{P}}^{R,K}_x$ is an infinitely divisible measure on the path space of $Y$ and (\ref{defi-nmeasure}) is a 'L\'evy-Khinchine' formula in which $\mbb{N}_x^{R,K}$ plays the role of the L\'evy measure.  In this sense, $\mbb{N}_x^{R,K}$ can be considered as the rate at which $\tilde{\mbf{P}}^{R,K}$-superdiffusions with infinitesimally small initial mass contribute to a unit mass at position $x$. 
Further we define, for $n \geq 2$, $x\in (0,K)$, 
\bex
\rho_n(dy,x) = \frac{    \beta w_K(x)^2 \delta_0(dy) \mbf{1}_{\{n=2\}}+ w_K(x)^n \frac{y^n}{n ! } e^{w_K(x) y} \Pi(dy)     }{q_n^{B,K}(x) w_K(x) \beta^{B,K}(x)} , 
\ex 
which will turn out to be the distribution of the initial mass of the immigrating $\tilde{\mbf{P}}^{R,K}$-superdiffusions at branch points of the backbone on the event of $n$ offspring. 
\begin{defi}\label{defi-dressedtreesbm}
Let $K>K_0$ and $\nu\in\mcal{M}_a(0,K)$. Let $X^B=(X_t^B, t\geq 0)$ be a $\mbf{P}^{B,K}$- branching diffusion with initial configuration $\nu$.
Suppose $I^{\mbb{N}^{R,K}}=(I^{\mbb{N}^{R,K}}_t, t\geq 0)$, $I^{\tilde{\mbf{P}}^{R,K}}=(I^{\tilde{\mbf{P}}^{R,K}}_t, t\geq 0)$ and $I^\rho=(I^\rho_t,t\geq 0)$ are three immigration processes (defined below) which are, conditionally on $X^B$, independent of each other. Then we define the process ${Y}^D=({Y}_t^D, t\geq 0)$ by
\bex
{Y}^D_t = I_t^{\mbb{N}^{R,K}} + I_t^{\tilde{\mbf{P}}^{R,K}} + I_t^\rho , \ \ \ t\geq 0
\ex
and denote its law by $\tilde{\mbf{P}}^{D,K}_\nu$.\\
The immigration processes are constructed as follows: \\
(i) Continuous immigration: The process  $I_t^{\mbb{N}^{R,K}} = (I_t^{\mbb{N}^{R,K}}, t \geq  0)$ is defined as \bex
 I_t^{\mbb{N}^{R,K}} =  \sum_{u\in \mcal{T}^B} \sum_{\tau_u^B \wedge t \leq s < \sigma_u^B \wedge t} Y_{t-s}^{(1,u,s)}, \ \ \ t\geq 0,
\ex
where, given $X^B$, independently for each $u\in\mcal{T}^B$ such that $\tau_u^B<t$, the processes $Y^{(1,u,s)}$ are countable in number and correspond to Poissonian immigration along the space-time trajectory $\{(x_u^B(s),s): s \in (\tau_u^B, \sigma_u^B]\}$ with rate  $2\beta ds \times d\mbb{N}_{x_u^B(s)}^R$.\\
(ii) Discontinuous immigration: The process  $I_t^{\tilde{\mbf{P}}^{R,K}} = (I_t^{\tilde{\mbf{P}}^{R,K}}, t \geq  0)$ is defined as \bex
 I_t^{\tilde{\mbf{P}}^{R,K}} =  \sum_{u\in \mcal{T}^B} \sum_{\tau_u^B \wedge t \leq s < \sigma_u^B \wedge t} Y_{t-s}^{(2,u,s)}, \ \ \ t\geq 0,
\ex
where, given $X^B$, independently for each $u\in\mcal{T}^B$ such that $\tau_u^B<t$, the processes $Y^{(2,u,s)}$ are  countable in number and correspond to Poissonian immigration along the space-time trajectory $\{(x_u^B(s),s): s \in (\tau_u^B, \sigma_u^B]\}$ with rate 
$ds\times \int_0^\infty y \exp\{-w_K(x_u^B(s)) y\} \ \Pi(dy) \times dP^{R,K}_{y\delta_{x_u^B(s)}}$.\\ 
(iii) Immigration at branch points: The process $I^\rho=(I^\rho_t, t\geq 0)$ is defined as \bex
I_t^\rho = \sum_{u\in \mcal{T}^B} \mbf{1}_{\{\sigma_u^B\leq t\}} Y_{t-\sigma_u^B}^{(3,u)}, \ \ \ t\geq 0,
\ex
where, given $X^B$, independently for each $u\in\mcal{T}^B$ such that $\sigma_u^B\leq t$, $Y^{(3,u)}$ 
is an independent copy of  $(Y,\tilde{\mbf{P}}^{R,K}_{Y_u\delta_{x_u^B(\sigma_u)}})$ issued at space-time position $({x_u^B(\sigma_u)},\sigma_u)$. At a branch point of $u$ with $n\geq 2$ offspring the initial mass $Y_u$ is distributed according to $\rho_n(dy,x_u^B(\sigma_u))$.
\end{defi}

\begin{theo}[Backbone decomposition]\label{theo-backbonesbm}
For $K>K_0$ and $\eta \in \mcal{M}_f[0,K]$, let $Y^R=(Y_t^R,t\geq 0)$ be an independent copy of $(Y,\tilde{\mbf{P}}^{R,K}_\eta)$. Suppose that $\nu$ is a Poisson random measure on $(0,K)$ with intensity $w_K(x) \eta(dx)$. Let  $(Y^D,\tilde{\mbf{P}}^{D,K}_\nu)$ be the process constructed in Definition \ref{defi-dressedtreesbm}. Define the process $\tilde{Y}=(\tilde{Y}_t,t\geq 0)$ by
\be\label{eq-backbonedecomp}
\tilde{Y}_t= Y^R_t + Y^{D}_t, \ \ t \geq 0,
\e
and denote its law by $\tilde{\mbf{P}}^K_\eta$.
Then the process $(\tilde{Y},\tilde{\mbf{P}}^K_\eta)$ is Markovian and equal in law to $(Y,\tilde{P}_\eta^K)$. 
\end{theo}
The proof of Theorem \ref{theo-backbonesbm} is a simple adaptation of the proofs of  Theorem 1 and 2 in \cite{bkm} and  therefore omitted. In principle it should be possible to give a proof analogous to the one presented in Section \ref{sec-backbone}, using that  $(Y,\tilde{P}_\eta^K)$ 
 conditioned on non-extinction arises from the martingale change of measure
\bex
\left.\frac{d\tilde{\mbf{Q}}^{K}_\nu}{d\tilde{P}^K_\nu}\right|_{\tilde{\mcal{F}}_t} = \frac{1-e^{-\la w_K, Y_t \ra}}{1-e^{- \la w_K ,\nu \ra}}, \ \ \ t \geq 0,
\ex
and showing that $(Y,\tilde{\mbf{Q}}^{K}_\nu)$ agrees in law with the process $(Y^D,\tilde{\mbf{P}}^{D,K}_\nu)$ of Definition \ref{defi-dressedtreesbm}. \\
The analogy between the $P^K$-branching diffusion and the 
$\tilde{P}^K$-super-diffusion indicates that there is a quasi-stationary limit result equivalent to Theorem \ref{theo-quasistationary}. 
Let us begin with constructing the analogue of the process $(X^*,P^*)$ in Theorem \ref{theo-quasistationary} for the superdiffusion setting. We introduce the family of $\mbb{N}$-measures now associated with the laws $(\tilde{P}_{\delta_x}^{K_0}, x\in[0,K_0])$.  Consider the family $\{\mbb{N}_x^{K_0},x\in [0,K_0]\}$  satisfying 
\bex
\mbb{N}_x^{K_0}(1-\exp{\la f,Y_t\ra}) = - \log \tilde{E}^{K_0}_{\delta_x}(e^{-\la f, Y_t\ra }), 
\ex
for $f\in B^+(0,K)$, $t\geq 0$. \\
Let $\eta\in\mcal{M}_f(0,K)$. Suppose $\xi^*=(\xi^*_t,t\geq 0)$ is a Brownian motion conditioned to stay in $(0,K_0)$ with initial position $x$ distributed according to 
\be\label{eq-distributioninitialparticle}
\frac{\sin(\pi x/K_0) e^{\mu x}}{\int_{(0,K_0)} \sin(\pi z/K_0) e^{\mu z} \ \eta(dz)} \eta(dx), \ \ x\in(0,K_0).
\e
Let $I_t^{\mbb{N}^{K_0}}=(I_t^{\mbb{N}^{K_0}},t\geq 0)$ and $I^{\tilde{P}^{K_0}}=(I^{\tilde{P}^{K_0}}_t,t\geq 0)$ be two immigration processes (defined below) which, conditionally on $\xi^*$, are independent of each other.  Then we define the process $Y^S=(Y_t^S, t\geq 0)$ by
\be\label{eq-dressedspine}
Y^S_t =  I_t^{\mbb{N}}+ I^{\tilde{P}^{K_0}}_t, \ \ t\geq 0.
\e
The immigration processes $I^{\mbb{N}^{K_0}}$ and $I^{\tilde{P}^{K_0}}$ are defined pathwise as follows. \\
(i) Continuous immigration: The process $I_t^{\mbb{N}^{K_0}}=(I_t^{\mbb{N}^{K_0}},t\geq 0)$ is defined as
\bex
I_t^{\mbb{N}^{K_0}} = \sum_{s\leq t } Y^{(\mbb{N},s)}_{t-s}, \ \ t\geq 0,
\ex
where, given  $\xi^*$, the processes $Y^{(\mbb{N},s)}$ are countable in number and correspond to Poissonian immigration along the space-time trajectory $\{(\xi^*_s,s): s \geq 0\}$ with
rate $2 \beta ds\times d\mbb{N}_{\xi_s^*}^{K_0}$; \\
(ii) Discontinuous immigration: The process
$I^{\tilde{P}^{K_0}}=(I^{\tilde{P}^{K_0}}_t,t\geq 0)$ is defined as
\bex
I^{\tilde{P}^{K_0}}_t = \sum_{s\leq t } Y^{(\tilde{P}^{K_0},s)}_{t-s}, \ \ t\geq 0,
\ex
where, given  $\xi^*$, the processes $Y^{(\tilde{P}^{K_0},s)}$ are countable in number and correspond to Poissonian immigration along the space-time trajectory $\{(\xi^*_s,s): s \geq 0\}$ with $ds\times \int_0^\infty y \Pi(dy)\times \tilde{P}^{K_0}_{y\delta_{\xi^*_s}}$.\\
Then define the process $Y^*=(Y^*_t, t\geq 0)$ by setting
\be\label{defi-evansimmortalparticlerep}
{Y}_t^* = Y'_t +Y^S_t, \ \ t\geq 0,
\e
where $Y'$ is an independent copy of $(Y,\tilde{P}^{K_0}_\eta)$.
We denote  the law of $Y^*$ by $\tilde{P}^*_\eta$.  The evolution of $Y^*$ under $\tilde{P}^*$ can thus be seen as a path-wise description of  Evans' immortal particle picture in \cite{evans} for the critical width $K_0$; for a similar construction of Evans' immortal particle picture see Kyprianou et al. \cite{klmr}. \\
Further, we note that $(Y^*,\tilde{P}_\eta^{K_0})$ has the same law us $Y$ under the measure which has martingale density $\tilde{Z}^{K_0}(t)$ of (\ref{eq-spinechangeofmeasuresbm}) with respect to $\tilde{P}_\eta^{K_0}$; for similar results see for instance Engl\"ander and Kyprianou \cite{englaenderkyprianou}, Kyprianou et al. \cite{klmr}  and  Liu et al. \cite{liuetal}.

\begin{theo}\label{theo-quasistationarysbm}
Let $K>K_0$ and $\eta\in\mcal{M}_f[0,K_0]$. For a fixed time $t\geq 0$, the law of $Y_t$ under the measure $\lim_{K\downarrow K_0} \tilde{P}_\eta^K(\cdot |\lim_{t\to\infty} ||Y_t|| >0)$ is equal to $Y_t^*$ under $\tilde{P}^*_\eta$.
\end{theo}

\begin{proof}[Sketch of the proof of Theorem \ref{theo-quasistationarysbm}]
By Theorem \ref{theo-backbonesbm}, $(Y,\tilde{P}_\eta^K)$ is equal in law to $(\tilde{Y},\tilde{\mbf{P}}_\eta^K)$. The latter is equal in law to $(Y,\tilde{\mbf{Q}}^K_\eta)$ where
\bex
\left.\frac{d\tilde{\mbf{Q}}^{K}_\eta}{d\tilde{P}^K_\eta}\right|_{\tilde{\mcal{F}}_t} = \frac{1-e^{-\la w_K, Y_t \ra}}{1-e^{-\la w_K , \eta \ra}}, \ \ \ t \geq 0.
\ex
The uniform asymptotics for $w_K$ in Theorem \ref{theo-survivalforsbm} let us conclude that
\bex
\lim_{K\downarrow K_0} \frac{1-e^{-\la w_K, Y_t \ra}}{1-e^{-\la w_K , \eta \ra}} 
&=& \lim_{K\downarrow K_0} \frac{\la w_K, Y_t \ra}{\la w_K, \eta \ra}\\
& = & \frac{\int_0^{K_0} \sin(\pi x /K_0) e^{\mu x} \ Y_t(dx)  }{\int_0^{K_0} \sin(\pi x/K_0) e^{\mu x} \ \mu(dx) \ra}
= \frac{\tilde{Z}^{K_0}(t)}{\tilde{Z}^{K_0}(0)},
\ex
where $\tilde{Z}^{K_0}$ is the martingale in (\ref{eq-spinechangeofmeasuresbm}). As mentioned before, the law of $Y$ under a change of measure with $\tilde{Z}^{K_0}$ is equal to  $(Y^*,\tilde{P}_\eta^{K_0})$.
\end{proof}

\section*{Acknowledgement}
We thank B. Derrida and E. A\"id\'ekon for a number of extremely useful and insightful discussions.

\bibliographystyle{alpha}
\newcommand{\etalchar}[1]{$^{#1}$}

\end{document}